\documentclass[12pt,reqno]{amsart}

\usepackage{amsfonts,amsmath,amsthm, amssymb}
\usepackage{graphicx}
\usepackage{here} 
\usepackage{cases}
\usepackage{caption}
\usepackage[usenames]{color}
\usepackage{enumerate} 

\theoremstyle{plain}
\newtheorem{thm}{Theorem}[section]

\newtheorem{defn}{Definition}

\newtheorem{lem}[thm]{Lemma}
\newtheorem{cor}[thm]{Corollary}
\newtheorem{prop}[thm]{Proposition}
\theoremstyle{remark}
\newtheorem{rem}{\bf{Remark}}
\numberwithin{equation}{section}

\newcommand{\N}{\mathbb{N}}

\newcommand{\R}{\mathbb{R}}
\newcommand{\bS}{\mathbb{S}}

\newcommand{\cF}{\mathcal{F}}

\newcommand{\USC}{{\rm USC\,}}
\newcommand{\LSC}{{\rm LSC\,}}
\newcommand{\Li}{L^{\infty}}
\newcommand{\Lip}{{\rm Lip\,}}

\newcommand{\al}{\alpha}
\newcommand{\gam}{\gamma}
\newcommand{\del}{\delta}
\newcommand{\ep}{\varepsilon}
\newcommand{\kap}{\kappa}
\newcommand{\lam}{\lambda}
\newcommand{\sig}{\sigma}

\newcommand{\Gam}{\Gamma}

\newcommand{\ol}{\overline}
\newcommand{\ul}{\underline}
\newcommand{\pl}{\partial}
\newcommand{\supp}{{\rm supp}\,}
\newcommand{\inter}{{\rm int}\,}

\newcommand{\Div}{{\rm div}\,}

\newcommand{\tr}{{\rm tr}\,}

\begin{document}
\title[On asymptotic speed of solutions to MCF equations with driving and source terms]
{On asymptotic speed of solutions to\\ 
level-set mean curvature flow equations\\ 
with driving and source terms}


\date{\today}

\thanks{
The work of YG was partially supported by Japan Society for the Promotion of Science (JSPS) through grants KAKENHI \#26220702, \#23244015, \#25610025.
The work of HM was partially supported by KAKENHI \#15K17574, \#26287024, \#23244015.
The work of HT was partially supported by
NSF grant DMS-1361236.}

\author[Y. Giga]{Yoshikazu Giga}
\address[Y. Giga]{
Graduate School of Mathematical Sciences, 
University of Tokyo 
3-8-1 Komaba, Meguro-ku, Tokyo, 153-8914, Japan}
\email{labgiga@ms.u-tokyo.ac.jp}

\author[H. Mitake]{Hiroyoshi Mitake}
\address[H. Mitake]{
Institute for Sustainable Sciences and Development,
Hiroshima University 1-4-1 Kagamiyama, Higashi-Hiroshima-shi 739-8527, Japan.}
\email{hiroyoshi-mitake@hiroshima-u.ac.jp}

\author[H. V. Tran]{Hung V. Tran}
\address[H. V. Tran]
{Department of Mathematics,
University of Wisconsin,
480 Lincoln Dr.,
Madison, WI 53706, USA.}
\email{hung@math.wisc.edu}



\begin{abstract}
We investigate a model equation in the crystal growth, which is 
described by a level-set mean curvature flow equation  
with driving and source terms. 
We establish the well-posedness of solutions, and study 
the asymptotic speed. 
Interestingly, a new type of nonlinear phenomena in terms of asymptotic speed of solutions appears, 
which is very sensitive to the shapes of source terms. 
\end{abstract}

\maketitle
\tableofcontents


\section{Introduction}
\subsection{Problem and Background}
In this paper, we study a level-set forced mean curvature flow equation 
motivated by a crystal growth phenomenon described below. 
The crystal grows in both vertical and horizontal directions. 
The vertical direction growth is stimulated by a nucleation, and the horizontal one is given 
by a surface evolution. 
We assume further that the surface evolution is 
described by the mean curvature with a constant force. 
Under these assumptions, the equation of interest is
\begin{equation}
{{\rm(C)}\qquad} \label{eq:1}
\begin{cases}
\displaystyle 
u_t-\left(\Div\left(\frac{Du}{|Du|}\right)+1\right)|Du|=c \mathbf{1}_E
 \quad &\text{in} \ \R^n\times(0,\infty), \\
u(\cdot,0)=0 \quad &\text{on} \ \R^n,
\end{cases}
\end{equation}
where $c>0$ and $E\subset \R^n$ are a given constant and a compact set respectively, and 
\[
\mathbf{1}_{E}(x):=
\left\{
\begin{array}{ll}
1 & 
\quad \textrm{if} \ x\in E, \\
0 & 
\quad \textrm{if} \ x\not\in E.  
\end{array}
\right. 
\]
Here for $(x,t)\in \R^n \times [0,\infty)$, $u(x,t)$ is the height of the crystal at location $x$ at time $t$.
\vspace*{1cm}
\hspace*{-1cm}

\begin{tabular}{cc}
\begin{minipage}{0.45\hsize}
\begin{center}
\hspace*{-1cm}
\includegraphics[width=6cm, bb=0 0 900 600]{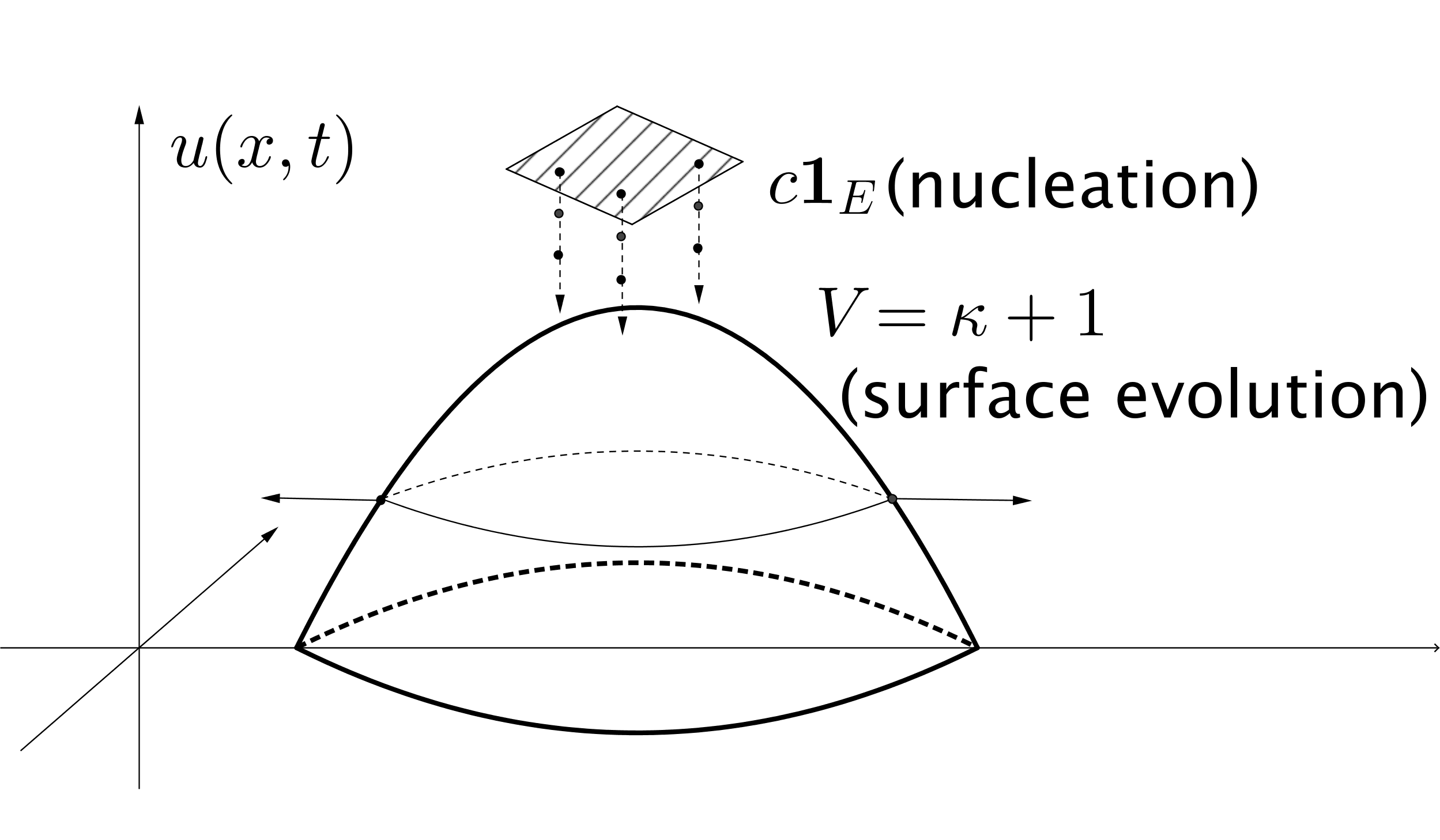}
\end{center}
\hspace*{2cm}
 \end{minipage}
\hspace*{1cm}
\begin{minipage}{0.47\hsize}
\vspace*{-1cm}
\begin{flushleft}
\noindent
Hypotheses: \\
1. There is a source term $c\mathbf{1}_E$ which stimulates 
the nucleation. \\
2. Each level set evolves following the law $V=\kappa+1$, 
where $\kappa$ is its mean curvature in the direction of the outer normal 
vector.
\end{flushleft}
\end{minipage}
\end{tabular}
\begin{center}
\captionof{figure}{Image of the crystal growth}
\end{center}

Let us explain the background of our model from the theory of crystal growth \cite{BCF}.
 We consider a perfectly flat surface of a crystal immersed in a supersaturated media.
 The crystal grows by catching adatoms.
 Assume that there are no dislocations on the surface so that no spiral growth is expected.
 There are several models to explain the growth of a perfectly flat crystal surface and their theories are often called two-dimensional nucleation growth theory \cite{OR,M}.
 They are roughly classified into single nucleation growth and multinucleation growth.
 Single nucleation growth is easy to explain.
 A nucleation starts somewhere and it spreads across the surface at an infinity velocity and the surface becomes flat again and waits next nucleation.
Multinucleation growth model was originally introduced by Hillig \cite{H} and developed in many ways especially to calculate the growth rate of the crystal surface, e.g.\ \cite{YK}.
There are several multinucleation growth models in the literature including ours, 
which is considered as a kind of the birth and spread model \cite[Section 2.6]{OR} originally proposed by Nielsen \cite{N}.

In the birth and spread model, an island of a layer of width $d$ of molecules, which looks like a pancake, is attached on the crystal surface
at the first nucleation.
 This pancake-like shape spreads to the surface with a finite speed
\begin{equation}\label{surface-ev}
V=v_\infty(\rho_c \kappa+1),  
\end{equation}
where $V$ and $\kappa$, respectively, denote the outer normal velocity of the lateral part of the island, and the outward curvature of the curve bounding the set in the plane where the island and the original crystal surface touches.
 The constants $\rho_c>0$ and $v_\infty>0$ are the critical radius and the step velocity \cite{BCF}, respectively. 
The second nucleation will start either on the top of the island or on the crystal surface and it also spreads afterwards.
The place where each nucleation starts is often randomly distributed if the concentration in supersaturated media is uniform.
 These phenomena are experimentally observed in the literature \cite{SZNYF} but it is less often than spiral growth.

Our model is considered as a continuum limit of a kind of the birth and spread model.
 Actually, from the continuum point of view, it is derived heuristically from the Trotter-Kato approximation (see Section \ref{sec:TK} for more details). 
 Let $E \subset \R^n$ be a set where nucleation starts  with supersaturation $c$.
 Let $u$ denote the height of the crystal surface, and initially the crystal surface is perfectly flat, i.e, $u(\cdot,0)\equiv0$.
 We assume that outside $E$ there is no supersaturation or no way to nucleate.
 Within very short time $d=\Delta t>0$ the crystal island forms and the shape of crystal becomes the graph of $d1_E$, i.e.\ $u(\cdot,d)=d1_E$.
During next short time $\Delta t$ this island spreads horizontally to the crystal surface with the velocity $V=v_\infty(\rho_c \kappa+1)$ 
and we get a profile $u(\cdot,2d)=d1_{E(\Delta t)}$, where $E(s)$ denotes the solution of the above eikonal-curvature flow \eqref{surface-ev} starting from $E$ at time $s>0$.
 In the next stage another nucleation starts again in $E$ and all pancake shapes spread with the same eikonal-curvature flow.
 By the repetition of this process we obtain a solution of the Trotter-Kato approximation of our original problem.
 Our original equation is its limit as $\Delta t \to 0$.
 In fact, equation \eqref{eq:1} is obtained when $v_\infty=1$ and $\rho_c=1$.
 
In the literature, it is sometimes assumed as in \cite[Section 2.6]{OR} that $\rho_c=0$ since $\rho_c$ is very small, and in this case, 
our continuum model is reduced to the Hamilton-Jacobi equation with discontinuous source term studied recently by Giga and Hamamuki \cite{GH},
 which justifies a proposed solution of T. P. Shulze and R. V. Kohn \cite{SK} modeling spiral growth phenomena consisting of a pair of screw dislocations with opposite orientation.
 In our model we take the curvature effect into account so that the model is more realistic. 
Indeed, due to the curvature effect, an interesting physical phenomenon can be explained, which is new and intuitive. 
Moreover, from mathematical point of view, we will face many serious difficulties 
because of the appearance of the curvature term.

\subsection{Main results}
Mathematically, equation (C) has two main parts. 
One comes from the nucleation described by the right hand side of 
(C), and the other one comes from the surface evolution. 
In general, the surface evolution of each level set 
is not only nonlinear but also nonhomogeneous and not monotone.
This makes the interaction between the nucleation and the surface evolution extremely nonlinear.
In a sense, in order to understand the behavior of solutions of (C), 
we need to understand the double nonlinear effects coming from the surface evolution and the interaction,
which will be explained more clearly in Sections \ref{sec:TK} and \ref{sec:rough-est}.

In this paper, we first establish the well-posedness of solutions to (C) as the 
equation is not so standard because of the appearance of the discontinuous source term in the right hand side. 
In particular, we show that the maximal solution is unique.
Our main goal then  is to study the asymptotic growth speed of the maximal solution $u$ to (C), i.e., 
\[
\lim_{t\to\infty}\frac{u(x,t)}{t}, 
\]
which describes the speed of the crystal growth. 
It is important noting that
the growth of $u$ can be seen in a heuristic way 
via the Trotter-Kato product formula in Section \ref{sec:TK}
which helps a lot in analyzing the double nonlinear effects.

In a specific case where $E=\ol{B}(0,R_0)$ for some given $R_0>0$, 
the growth is completely understood and is analyzed in Section \ref{sec:sphere}.
Here $B(x,R)$ denotes the open ball of radius $R>0$ centered at $x \in \R^n$.
The analysis could be done in an explicit way here as the surface evolution is homogeneous.
More precisely, balls remain balls after the surface evolution.
This also makes the interaction between the nucleation and the surface evolution simple and clear.
The governing equation (C) becomes a first order equation but it is noncoercive. 
As studied in \cite{GLM1, GLM2} the asymptotic speed may depend on the place when the equation is noncoercive. 
See also \cite{YGR} for its physical background.
We prove the following theorem through Propositions \ref{prop:case1}, \ref{prop:case2} and \ref{prop:case3}.

\begin{thm}\label{thm:main1}
Assume that $E=\ol{B}(0,R_0)$ for some $R_0>0$ fixed.
Let $u$ be the maximal solution of {\rm (C)}. The followings hold
\begin{itemize}
\item[(i)] If $R_0<n-1$, then $u$ has the formula \eqref{func:u-case1} and is bounded on $\R^n \times [0,\infty)$. In particular,
\[
\lim_{t \to \infty} \frac{u(x,t)}{t}=0 \quad \text{uniformly for} \ x \in \R^n.
\]

\item[(ii)] If $R_0>n-1$, then $u$ has the formula \eqref{func:u-case2} and
\[
\lim_{t \to \infty} \frac{u(x,t)}{t}=c \quad \text{locally uniformly for} \ x \in \R^n.
\]

\item[(iii)] If $R_0=n-1$, then $u=ct\mathbf{1}_{\ol{B}(0,n-1)}$. In particular,
\[
\lim_{t \to \infty} \frac{u(x,t)}{t}=c \quad \text{uniformly for} \ x \in \ol{B}(0,n-1),
\]
and
\[
\lim_{t \to \infty} \frac{u(x,t)}{t}=0 \quad \text{ uniformly for} \ x \in \R^n \setminus \ol{B}(0,n-1).
\]
\end{itemize}
\end{thm}
Some more general results for the case of inhomogeneous source terms which are radially symmetric are analyzed in Section \ref{subsec:opt} as well. 
See Theorem \ref{thm:h1} and Remark \ref{rem:general}. 
\smallskip

In the general case where $E$ is not radially symmetric, 
it turns out that a new type of nonlinear phenomena appears. 
Very roughly speaking, if the nucleation site $E$ is small enough, 
then the solution of (C) does not grow up globally as $t\to\infty$.
On the other hand, if the set $E$ is big enough, 
then the solution of (C) grows up locally uniformly in $\R^n$ as $t\to\infty$ 
with an asymptotic speed $c$, which is the rate of nucleation. 
If the set $E$ is of middle size in a sense, it seems that the asymptotic 
speed depends on its shape in a very delicate and sensitive way. 

It is worthwhile emphasizing here that we find such phenomena as a fact of 
experiment in the crystal growth.
In the experiment the set $E$ corresponding the places where adatoms are sprayed on crystal surfaces. 
Such a situation is quite popular for growth of metals in manufacturing technology of semiconductors although spreading mechanism has other effects. See e.g., \cite{ZW}. 
For such phenomena as well as quantum dots the curvature effect may not be neglected in two-dimensional setting.

We provide a framework to get estimates of growth rate in Sections \ref{sec:TK} and \ref{sec:rough-est}.
We then choose a representative case where 
$E$ is of square shape in $\R^2$ to study in details in Section \ref{sec:square}.
We establish the following theorem through Propositions \ref{prop:cor-square} and \ref{prop:middle}.

\begin{thm}\label{thm:main2}
Assume that $n=2$ and $E=\{(x_1,x_2)\,:\,|x_i| \leq d, \ i=1,2\}$ for some $d>0$ fixed.
Let $u$ be the maximal solution of {\rm (C)}. The followings hold
\begin{itemize}
\item[(i)] If $d<1/\sqrt{2}$, then $u$ is bounded on $\R^2 \times [0,\infty)$. In particular,
\[
\lim_{t \to \infty} \frac{u(x,t)}{t}=0 \quad \text{uniformly for} \ x \in \R^2.
\]

\item[(ii)] If $d>1$, then 
\[
\lim_{t \to \infty} \frac{u(x,t)}{t}=c \quad \text{locally uniformly for} \ x \in \R^2.
\]

\item[(iii)] If $d=1$, then
\[
\lim_{t \to \infty} \frac{u(x,t)}{t}=c \quad \text{uniformly for} \ x \in \ol{B}(0,1).
\]
\item[(iv)] If $1/\sqrt{2}<d<1$, then there exist $\al,\beta$ such that $0<\al<\beta<c$ and
\[
\al \leq \liminf_{t \to \infty} \frac{u(x,t)}{t} \leq \limsup_{t \to \infty} \frac{u(x,t)}{t} \leq \beta
\quad \text{locally uniformly for} \ x \in \R^2.
\]
\end{itemize}
\end{thm}

Parts (i)--(iii) of Theorem \ref{thm:main2} are obtained straightforwardly by using Theorem \ref{thm:main1} and the comparison principle.  
To prove part (iv) of Theorem \ref{thm:main2}, a first important step is to understand the behavior of the level 
sets of the top and bottom of solutions to (C). 
We study this by using a set theoretic approach (see \cite[Chapter 5]{G} for instance) 
in Section \ref{sec:rough-est}. 
This perspective gives us  rough estimates (Theorems \ref{thm:upper-esti}, \ref{thm:lower-esti}) 
on the behavior of the height of solutions to (C), 
but this is not enough to obtain the precise behavior of $\lim_{t \to \infty} u(x,t)/t$. 
Indeed, we have not yet been able to get the precise behavior $\lim_{t \to \infty} u(x,t)/t$ in part (iv) of Theorem \ref{thm:main2},
which is the case where $E$ is of middle size.

\medskip
We conclude this Introduction to give some related works studying large-time asymptotic behavior of solutions of problems which are non-coercive or of second order problem. The list is not exhaustive at all. As an example of non-coercive Hamilton-Jacobi equations the instability of flatness in crystal growth is discussed in \cite{YGR, GLM1, GLM2}, and 
the turbulent flame speed is studied in the context of G-equations in \cite{XY1, XY2}. 
In \cite{CN} the large-time behavior of solutions of mean curvature flow equations with driving force is studied.  

\medskip\noindent
\textbf{Acknowledgments.}
 The first two authors are grateful to Professor Etsuro Yokoyama and Professor Hiroki Hibino for their kind comments on the two-dimensional nucleation.


\section{Wellposedness}\label{sec:well}
In this section, we consider a little bit more general equation 
\begin{equation}
\label{eq:2}
\begin{cases}
\displaystyle
u_t-\left(\Div\left(\frac{Du}{|Du|}\right)+1\right)|Du|=f(x)
 \quad &\text{in} \ \R^n\times(0,\infty), \\
u(\cdot,0)=u_0 \quad &\text{on} \ \R^n,
\end{cases}
\end{equation}
where $f:\R^n\to[0,\infty)$ is a bounded function, and $u_0:\R^n\to[0,\infty)$ is a continuous with 
\begin{equation}\label{comp-supp}
\supp u_0, \ \supp f\subset B(0,R) \quad\text{for some} \ R>0. 
\end{equation}
This is an important assumption as,
for any $T>0$, we only deal with compactly supported solutions of \eqref{eq:2} and (C)
on $\R^n \times [0,T]$.
Notice that 
\[
\Div\left(\frac{Du}{|Du|}\right)|Du|=\tr\left[\left(I-\frac{Du\otimes Du}{|Du|^2}\right)D^2u\right], 
\]
where $I$ is the identity matrix of size $n$.
Set $\sig(p):=I-(p\otimes p)/|p|^2$, and 
\begin{equation}\label{func:H}
H(p,X):=-\tr\left[\sig(p)X\right]-|p| \quad
\text{for} \ (p,X)\in(\R^n\setminus\{0\})\times\bS^{n}, 
\end{equation}
where $\bS^n$ is the set of real symmetric matrices of size $n$.

We first recall the definition of viscosity solutions to 
an equation with discontinuous functions, which 
was introduced in \cite{I}.  
\begin{defn}[Viscosity solutions]
{\rm
Let $u:\R^n\times[0,\infty) \to \R$ be a locally bounded.
We say that $u$ is a viscosity subsolution of \eqref{eq:2} if  $u^\ast(\cdot,0)\le u_0$ on $\R^n$, and 
\[
\tau+H_\ast(p,X)\le  f^{\ast}(x_0) \quad\text{for all} \ (x_0,t_0)\in\R^n\times(0,\infty),  (p,X,\tau)\in{J}^{+}u^\ast(x_0,t_0). 
\]
We say that $u$ is a viscosity supersolution of  \eqref{eq:2} if 
$u_\ast(\cdot,0)\ge u_0$ on $\R^n$,  and
\[
\tau+H^\ast(p,X)\ge f_{\ast}(x_0) \quad\text{for all} \ (x_0,t_0)\in\R^n\times(0,\infty), \ (p,X,\tau)\in {J}^{-}u_\ast(x_0,t_0). 
\]
Here 
for a locally bounded function $h$ on $\R^m$ for $m\in\N$, 
we denote the upper semicontinuous envelope (resp., lower semicontinuous envelope) 
by $h^\ast$ and $h_\ast$ defined as  
$h^\ast(x):=\lim_{\del\to0}\sup\{h(y)\,:\, |x-y|\le\del\}$ and 
$h_\ast(x):=\lim_{\del\to0}\inf\{h(y)\,:\, |x-y|\le\del\}$, respectively,    
and we write $J^{+}u^{\ast}(x,t)$ and $J^{-}u_{\ast}(x,t)$ for the super and sub semijets of $u^{\ast}, u_{\ast}$ at $(x,t)\in\R^n\times(0,\infty)$, respectively. 

We say that $u$ is a viscosity solution of \eqref{eq:2} 
if it is both a viscosity subsolution and a viscosity supersolution of \eqref{eq:2}. 
}
\end{defn}

It is well-known that 
if the function $f$ on the right hand side of \eqref{eq:2} is continuous on $\R^n$, 
then the comparison principle and the uniqueness of solutions hold. 
See \cite{G} for instance. 
On the other hand, if we deal with discontinuous functions $f$ on the right hand side of \eqref{eq:2}, we lose the uniqueness of viscosity solutions in general. 
See \cite{GH} for some examples and observations of first order Hamilton-Jacobi equations with discontinuous source terms.  
We only have the comparison principle in a weak sense. 
The following result is standard in the theory of viscosity solutions, 
but we present it here to make the paper self-contained.

\begin{prop}[Weak Comparison Principle]\label{prop:comparison}
Fix $T>0$. Assume that $v\in \USC(\R^n\times[0,T])$ and $w\in \LSC(\R^n\times[0,T])$,  
which are compactly supported, i.e., 
\begin{equation}\label{compact-support}
v(x,t)=w(x,t)=0\quad\text{for all} \ x\in\R^n\setminus B(0,R_T), t\in[0,T] \ \text{and} \  
\text{some} \ R_T>0,  
\end{equation}
are a viscosity subsolution and a viscosity supersolution of 
\eqref{eq:2} with $f, g$ on the right hand side respectively, 
where $f$ and $g$ are locally bounded functions satisfying 
$f^{\ast} \leq g_{\ast}$ on $\R^n$. 
Then $v\le w$ on $\R^n\times[0,T]$. 
\end{prop}
\begin{proof}
%
We argue by contradiction and suppose that $\max_{\R^n\times[0,T]}(v-w)>0$. 
Then, there exists a small constant $\al>0$ such that 
for each $\ep>0$ sufficiently small, we have
\[
\max_{\substack{x,y\in\R^n\\t\in[0,T]}}
\left\{v(x,t)-w(y,t)-
\frac{|x-y|^4}{4\ep}-\frac{\al}{T-t}\right\}>0. 
\]
As $v(x,t)=w(x,t)=0$ for all $(x,t)\in(\R^n\setminus B(0,R_T))\times[0,T]$, 
the maximum is attained at 
$(x_\ep,y_\ep,t_\ep)\in B(0,R_T)^2\times(0,T)$ and by passing to a subsequence if necessary, we can assume 
$(x_\ep,y_\ep,t_\ep) \to (x_0,x_0,t_0)$ as $\ep \to 0$ for some $x_0\in\ol{B}(0,R_T), t_0\in[0,T]$.

In view of Ishii's lemma, for any $\rho>0$, 
there exist 
$(a_\ep,p_\ep,X_\ep)\in {J}^{2,+}v(x_\ep,t_\ep)$ and 
$(b_\ep,p_\ep,Y_\ep)\in {J}^{2,-}w(y_\ep,t_\ep)$ 
such that 
\begin{gather}\label{ishii-lem}
a_\ep-b_\ep=\frac{\al}{(T-t_\ep)^2}, \quad 
p_\ep=\frac{|x_\ep-y_\ep|^2(x_\ep-y_\ep)}{\ep}, \quad 
\left(
\begin{array}{cc}
X_\ep & 0 \\
0 & -Y_\ep 
\end{array}
\right)
\le 
A+\rho A^{2}, 
\end{gather}
where 
\begin{align*}
A:=&
\frac{1}{\ep}|x_\ep-y_\ep|^2 
\left(
\begin{array}{cc}
I & -I \\
-I & I 
\end{array}
\right)
+\frac{2}{\ep}
\left(
\begin{array}{cc}
(x_\ep-y_\ep)\otimes(x_\ep-y_\ep) & 
-(x_\ep-y_\ep)\otimes(x_\ep-y_\ep) \\
-(x_\ep-y_\ep)\otimes(x_\ep-y_\ep) & 
(x_\ep-y_\ep)\otimes(x_\ep-y_\ep)
\end{array}
\right).  
\end{align*}

The definition of viscosity solutions implies the following inequalities: 
\begin{equation}\label{ineq-1}
a_\ep+H_\ast(p_\ep,X_\ep)\le f^{\ast}(x_\ep), 
\quad\text{and}\quad
b_\ep+H^\ast(p_\ep,Y_\ep)\ge g_{\ast}(y_\ep).
\end{equation}
Note that \eqref{ishii-lem} implies $X_\ep\le Y_\ep$. 

In the case $p_\ep\not=0$, i.e., $x_\ep\not=y_\ep$, we have 
\[
H_\ast(p_\ep,X_\ep)-H^\ast(p_\ep,Y_\ep)=\tr\left[\sig(p_\ep)(Y_\ep-X_\ep)\right]\ge0. 
\]
In the case $p_\ep=0$, we have $x_\ep=y_\ep$. 
Due to \eqref{ishii-lem},  we have $A=0$, which implies $X_\ep\le0$ and $Y_\ep\ge0$. 
Thus, 
\[
H_{\ast}(p_\ep,X_\ep)\ge 
H_{\ast}(0,0)=0, 
\ \textrm{and} \ 
H^{\ast}(p_\ep,Y_\ep)\le 
H^{\ast}(0,0)=0. 
\]
In both cases, $H_\ast(p_\ep,X_\ep)-H^\ast(p_\ep,Y_\ep)\ge0$. 
Combine this with \eqref{ineq-1} to yield
\[
\frac{\al}{T^2}\le\frac{\al}{(T-t_\ep)^2}\le f^{\ast}(x_\ep)-g_{\ast}(y_\ep).
\]
Let $\ep \to 0$ to deduce that
\[
\limsup_{\ep\to0} (f^\ast(x_\ep)-g_\ast(y_\ep))\le 
\limsup_{\ep\to0}f^\ast(x_\ep)-\liminf_{\ep\to0}g_\ast(y_\ep)
\le (f^\ast-g_\ast)(x_0)\leq 0,
\] 
which is a contradiction. 
\end{proof}

In the case where we drop the curvature term, i.e., we consider 
the Hamilton--Jacobi equation with a discontinuous source term, 
if we additionally assume 
\begin{equation}\label{assume:f}
(f_\ast)^{\ast}=f^\ast \quad\text{on} \ \R^n, 
\end{equation}
then we can prove the uniqueness of viscosity solutions in the class of upper semicontinuous functions in the sense of \cite{I}. 
This can be done by using a control approach which is an analogue of \cite{BP}. 
On the other hand, as far as the authors know, there is no uniqueness results for \eqref{eq:2} under \eqref{assume:f}. 
Therefore, in this paper, 
we consider the maximal viscosity solutions of \eqref{eq:2}. 
To make it clear, we give its definition here. 
\begin{defn}[Maximal viscosity solutions]
{\rm
We say that $u$ is a {\it maximal} viscosity solution of \eqref{eq:2} if $u$ is a viscosity solution of \eqref{eq:2} satisfying 
\eqref{compact-support} and 
for every viscosity solution $v$ of \eqref{eq:2} satisfying \eqref{compact-support},
$ u \geq v$ on $\R^n\times[0,\infty)$.
}
\end{defn}

\begin{thm}[Existence and Uniqueness]\label{thm:existence}
There exists a unique maximal viscosity solution $u$ of \eqref{eq:2}.  
\end{thm}
\begin{proof}
Fix $T>0$.
For $k\in\N$ and $x\in \R^n$, define
\[
f^k(x):=\sup_{y\in \R^n} \left(f^{\ast}(y)-k|x-y|\right). 
\]
 It is straightforward that $f^k\in C(\R^n)$ and $f^k(x)\downarrow  f^\ast(x)$ pointwise as $k\to\infty$.  

By the standard theory of viscosity solutions, there exists a unique viscosity solution $u^{k} \in C(\R^n \times [0,T])$ 
of \eqref{eq:2} with the right hand side $f^k$, and by the comparison principle, we can easily prove $u^{k}(x)\downarrow u(x)$ for all $x\in\R^n$ as $k\to \infty$.
Furthermore,  in light of assumption \eqref{comp-supp},  there exists $R_T>0$ such that
\[
u^k(x,t)=u(x,t)=0 \quad \text{for all}\ (x,t) \in (\R^n \setminus B(0,R_T)) \times [0,T], \ \text{and} \ k \in \N.
\] 
Note that in view of this monotonicity 
\begin{equation}\label{uk-decreasing}
u=\inf_{k\in\N}u^k=\limsup{}^{\ast}_{k\to\infty}u^k,  
\end{equation}
where $\limsup{}^{\ast}$ is the upper half-relaxed limit. Thus, $u \in \USC(\R^n \times [0,T])$.

Clearly, $u^k$ is a supersolution of \eqref{eq:2} for each $k\in \N$,
which yields immediately that $u$ is also a supersolution of \eqref{eq:2} in view of the $\inf$-stability.
Thanks to \eqref{uk-decreasing},  $u$ is a subsolution of \eqref{eq:2} in view of the 
stability property of the upper half-relaxed limit for viscosity subsolutions.
Therefore, $u$ is a solution of \eqref{eq:2}. 
Moreover, $u$ is compactly supported on $\R^n \times [0,T]$.
 
Next, we prove that $u$ is continuous at $t=0$. 
By \cite[Lemma 4.3.4]{G}, 
there exists a viscosity subsolution $v \in C(\R^n \times [0,T])$ with a compact support of \eqref{eq:2} with $f=0$ on the right hand side.  
Thus, by the comparison principle, $v(y,t)-u_0(x)\le u(y,t)-u_0(x)\le u^k(y,t)-u_0(x)$ for all $x,y\in\R^n$, $t\in [0,T]$, 
which implies $u(y,t)\to u_0(x)$ as $(y,t)\to(x,0)$. 

Finally, we show that $u$ is the unique maximal viscosity solution of \eqref{eq:2}. 
Indeed, take any $v$ to be a viscosity solution of \eqref{eq:2}. By the comparison principle, $v^\ast \leq u^k$. 
Let $k \to \infty$ to deduce the desired result.
\end{proof}


\section{Heuristic observation}\label{sec:TK}
In this section, we give a formal argument in order to understand 
the behavior of solutions of (C). 
Our goal in this section is to explain intuitively with a geometric aspect
how the asymptotic average of solutions depends on the shape of $E$. 
This is basically the same as the derivation of the problem from physics explained in Introduction.

\subsection{Trotter-Kato product formula}
We consider the following double-step method:
\begin{numcases}
{{\rm (N)}\quad}
v_t=c\mathbf{1}_{E}
 & in $\R^n\times(0,\infty)$, \nonumber \\
v(\cdot,0)=u_0 & in $\R^n$, \nonumber 
\end{numcases}
and 
\begin{numcases}
{{\rm (P)}\quad}
w_t=\left(\Div\Big(\frac{Dw}{|Dw|}\Big)+1\right)|Dw| 
 & in $\R^n\times(0,\infty)$, \nonumber \\
w(\cdot,0)=u_0 & in $\R^n$. \nonumber 
\end{numcases}
We call (N) and (P) the \textit{nucleation problem} and the \textit{propagation problem}, respectively. 
We define the operators $S_1(t):\Li(\R^n)\to\Li(\R^n)$, and $S_2(t):\Lip(\R^n)\to\Lip(\R^n)$, respectively by  
\begin{equation}\label{TK-1}
S_1(t)[u_0]:=u_0+c\mathbf{1}_{E}t, \quad\text{and}\quad
S_2(t)[u_0]:=w(\cdot,t),   
\end{equation}
where $w$ is the unique viscosity solution of (P). 

For $x\in\R^n, \tau>0, i\in\N$, set 
\begin{equation}\label{TK-formula}
U^\tau(x,i\tau):=S_1(\tau)\big(S_2(\tau)S_1(\tau)\big)^{i}[u_0].  
\end{equation}
The function $U^\tau(x, i\tau)$ is called the \textit{Trotter-Kato product formula}, 
we can expect for $x\in \R^n$ and $t=i \tau>0$ fixed, 
\begin{equation}\label{TK:limit}
\lim_{i \to \infty} U^\tau(x,i\tau)= u(x,t) \quad\text{locally uniformly for} \ x\in\R^n, 
\end{equation}
under some condition, where $u$ is the ``solution" of (C). 
In the framework of the theory of viscosity solutions, Barles and Souganidis in \cite{BS} 
first proved \eqref{TK:limit}.  
Our situation here actually does not fit into the framework of \cite{BS},  
as we do not have the comparison principle for (C) because of the discontinuous 
source term $c\mathbf{1}_E$ on the right hand side of (C). 
Nevertheless, it is quite reasonable to assume that \eqref{TK:limit} holds 
in order to guess the behavior of the solution $u$ to (C).

In light of this, the behavior of $u(x,t)/t$ as $t\to \infty$
can be consider as the behavior of 
\begin{equation}\label{limit-TK}
\lim_{t \to \infty} \left( \lim_{\substack{\tau \to 0\\ i\tau =t}} \frac{U^\tau(x,i\tau)}{i \tau}\right). 
\end{equation}
The advantage of considering $U^\tau(x,i \tau)$ lies in the fact that its graph is 
a pyramid of finite number of steps of height $c\tau$.
The double-step method can then be described in a geometrical way as follows: 
\begin{itemize}
\item[\textbf{(N)}]
At each nucleation step, we drop from above an amount of $c\tau \mathbf{1}_E$ crystal
down to the pyramid with the assumption that the crystals are not sticky; 
\item[\textbf{(P)}]
At each propagation step, each layer of the pyramid evolves under a forced mean curvature flow ($V=\kappa+1$).
\end{itemize}

Let us emphasize that, in general, the growth of the pyramid is highly nonlinear. 
The reason comes from the fact that the behavior of each layer is extremely complicated, which will be pointed out 
in Section \ref{sec:rough-est} in more clearly. 
One particular layer can receive some amount of crystal in each nucleation step, then changes its shape 
in each propagation step.
Of course the layers change not only in a nonlinear way
but also in a nonhomogeneous way  in each propagation step.
Furthermore, the changes are not monotone (unlike the case $V=1$).
These affect the next nucleation step seriously as the receipt of crystals at each 
layer will change dramatically from time to time.
More or less, this says that the problem has  double nonlinear effects.

\subsection{Notations and Spherical symmetric case}\label{sec:obs1}
In this subsection, we make the analysis above clearer by a careful step by step analysis. 
In order to do so, we introduce notations below. 
For $A\subset\R^n$ and  $t>0$, let $\cF[A](t)$  be the solution to the surface evolution equation 
\[
{\rm(S)}\qquad
V=\kap+1 \ \text{on} \ \Gam(t) \quad \text{with} \quad  \Gam(0)=A.
\]
Fix $i\in\N$ and $\tau>0$. For $j\in\{1,\dots,i\}$ and $k\in\{1,\dots,j\}$, 
we define the sets $E_{\tau}(j,k)\subset\R^n$,
which are the layers of the pyramids, as follows: 
\begin{align*}
E_{\tau}(1,1):=&E,  \\
E_{\tau}(2,1):=&E\cup\cF[E_{\tau}(1,1)](\tau), \
E_{\tau}(2,2):=E\cap\cF[E_{\tau}(1,1)](\tau), \\ 
E_{\tau}(3,1):=&E\cup\cF[E_{\tau}(2,1)](\tau), \ 
E_{\tau}(3,2):=(E\cap\cF[E_\tau(2,1)](\tau))\cup\cF[E_{\tau}(2,2)](\tau), \\
E_{\tau}(3,3):=&E\cap\cF[E_{\tau}(2,2)](\tau),  \\
\vdots&
\end{align*}

Let us now use this system to investigate spherical symmetric cases, i.e., 
\[
E=\ol{B}(0,R_0). 
\]
It is worth pointing out that the spherical symmetric cases are easy to understand
 because of the fact that balls remain balls after the evolution under the forced mean curvature flow
$V=\kappa+1$. So the changes in shapes in each propagation step are homogeneous, which make 
behaviors of the nucleation steps and the pyramids extremely clear.
Our concern therefore is only whether $\ol{B}(0,R_0)$ grows or shrinks under the propagation step or not.
This leads to the distinction between the three cases: 
\begin{equation}\label{3case}
R_0<n-1, \quad
R_0>n-1,\quad \text{and} \quad
R_0=n-1.
\end{equation}

We first consider the case where $R_0<n-1$. 
As the curvature term is stronger than the force term, the surfaces 
start to shrink.  Thus, 
\begin{align*}
E_{\tau}(1,1)=&\ol{B}(0,R_0),  \\
E_{\tau}(2,1)=&\ol{B}(0,R_0), \
E_{\tau}(2,2)=\ol{B}(0,R(\tau)), \\ 
E_{\tau}(3,1)=&\ol{B}(0,R_0), \ 
E_{\tau}(3,2)=\ol{B}(0,R(\tau)), \
E_{\tau}(3,3)=\ol{B}(0,R(2\tau)),  \\
\vdots&
\end{align*}
where $R(t)$ is the solution of the ODE
\begin{numcases}
{}
\dot{R}(t)=-\frac{n-1}{R(t)}+1 
 & for $t>0$, \nonumber \\
R(0)=R_0. &  \nonumber 
\end{numcases}

On the other hand, if $R_0>n-1$, then the curvature term is weaker than the force term and 
the surfaces start to expand. Thus, 
\begin{align*}
E_{\tau}(1,1)=&\ol{B}(0,R_0),  \\
E_{\tau}(2,1)=&\ol{B}(0,R(\tau)), \
E_{\tau}(2,2)=\ol{B}(0,R_0), \\ 
E_{\tau}(3,1)=&\ol{B}(0,R(2\tau)), \ 
E_{\tau}(3,2)=\ol{B}(0,R(\tau)), \
E_{\tau}(3,3)=\ol{B}(0,R_0).  \\
\vdots&
\end{align*}
By using these observations, we can somehow understand the behavior of 
\eqref{limit-TK} in each cases of \eqref{3case}.

In the next section, we only consider the spherical case and 
rigorously derive the formula of $u$ as well as its large time average.

\section{Spherical symmetric case} \label{sec:sphere}
In this section, we study the case where 
\[
E=\ol{B}(0,R_0)
\quad\text{for some} \ R_0>0,   
\]
and investigate the large time average of the maximal viscosity solution $u$. 

It is reasonable to look for radially symmetric solution $u$ of (C) of the form
\[
u(x,t)=\phi(|x|,t)=\phi(r,t). 
\]
Then, 
\begin{align*}
&u_t=\phi_t, \  
Du=\phi_r \frac{x}{|x|}, \\
&
D^2u= 
\phi_{rr}\frac{x\otimes x}{|x|^2}
+\phi_r\frac{1}{|x|}\Big(I-\frac{x\otimes x}{|x|^2}\Big). 
\end{align*}
Plugging these into (C) to reduce it to the following initial value problem, 
which is basically a singular noncoercive Hamilton--Jacobi equation in $1-$dimension,  
\begin{equation*}
{{\rm(C')}\quad}
\begin{cases}
\displaystyle
\phi_t-\frac{(n-1)\phi_r}{r}-|\phi_r|=c\mathbf{1}_{[0,R_0]}(r) 
 \quad &\text{in} \ (0,\infty)\times(0,\infty), \\
\phi(\cdot,0)=0 \quad &\text{on} \ [0,\infty).
\end{cases}
\end{equation*}
From the next subsequences, we consider three cases divided in \eqref{3case}. 

\subsection{The case $R_0 < n-1$}\label{subsec:subcritical} 
In order to obtain the maximal viscosity solution, we approximate from above to 
get a decreasing sequence of supersolutions. 
Its limit will be the maximal viscosity solution once we prove that it is a subsolution. 

Fix $\ep>0$ sufficiently small such that $R_0+\ep< n-1$. 
We first solve a boundary value problem of a linear ordinal differential equation: 
\begin{equation*}
{{\rm(ODE)}\quad}
\begin{cases}
\displaystyle
\left(-\frac{(n-1)}{r}+1\right)\psi^{\ep}_r=cI^{\ep}(r) \quad \text{in} \ (0,R_0+\ep), \\
\psi^{\ep}(R_0+\ep)=0, 
\end{cases}
\end{equation*}
where 
\[
I^{\ep}(r):=
 \begin{cases}
 1 \quad &\text{for} \ r\in[0,R_0],\\
 \frac{(R_0+\ep)-r}{\ep} \quad &\text{for}  \ r\in[R_0, R_0+\ep],\\
 0 \quad &\text{for} \ r \in [R_0+\ep,+\infty).
 \end{cases}
\]
It is clear that, for $0\leq r \leq R_0+\ep$,
\[
\psi^{\ep}(r):=\int_0^r\frac{cI^{\ep}(s)}{-\frac{n-1}{s}+1}\,ds
-\int_0^{R_0+\ep} \frac{cI^{\ep}(s)}{-\frac{n-1}{s}+1}\,ds.
\] 
We set $\psi^{\ep}(r):=0$ for $r\ge R_0+\ep$, 
and extend $\psi^\ep$ to the whole $\R$ in a symmetric way 
(i.e., set $\psi^\ep(r)=\psi^\ep(-r)$ for all $r\leq 0$). 
Denote $u^\ep\in C(\R^n\times[0,\infty))$ by 
\begin{equation}\label{func:u-case1-ap}
 u^\ep(x,t):=\min\left\{\psi^\ep(|x|),ct\right\}. 
\end{equation}
\begin{center}
\includegraphics[width=80mm]{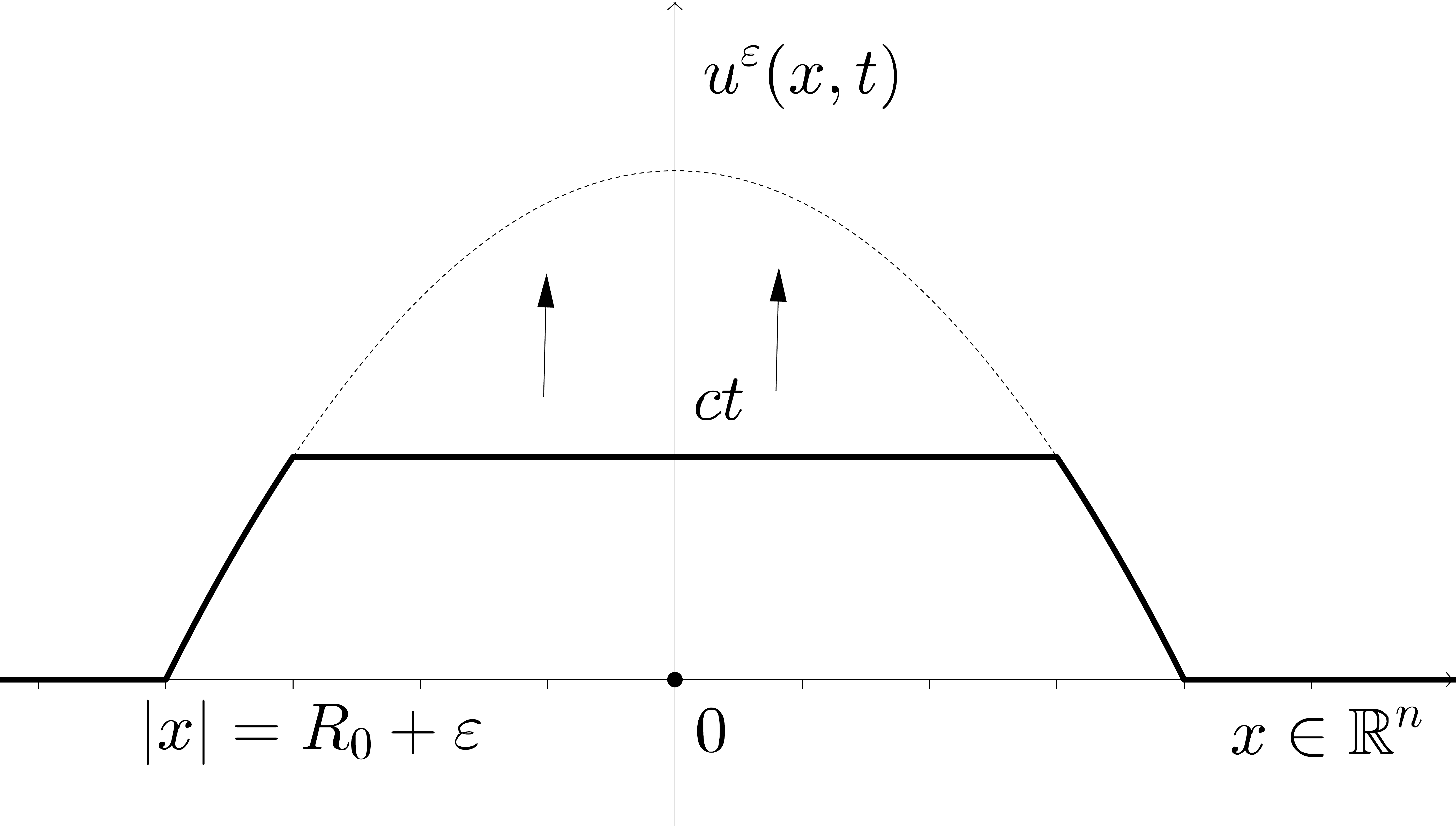}
\captionof{figure}{Picture of $u^\ep$ in Case 1}
\end{center}

\begin{lem}\label{lem:psi-ep}
The function $u^\ep$ is a viscosity supersolution of \eqref{eq:2}
for $g(x)= cI^\ep(|x|)$.
\end{lem}

\begin{proof}
 The claim is clear for $(x,t) \in B(0,R_0+\ep)\times (0,\infty)$
as $u^\ep$ is the minimum of two supersolutions there.
Also there is nothing to check in case $|x|>R_0+\ep$ as $u^\ep=g=0$ there.

We only need to check carefully the case where $|x|=R_0+\ep$. 
It is worth to mention first that for any $t_0>0$, 
we always have that $u^\ep(x,t)=\psi^\ep(|x|)$ for $x$ in a neighborhood of $\partial B(0,R_0+\ep)$
and $t \in (t_0/2,t_0+1)$. 
In other words, $u^\ep$ does not change with respect to time in this neighborhood.
Take $\phi \in C^2(\R^n)$ to be a test function such that $\psi^\ep - \phi$ has a strict minimum
at $x_0 \in \partial B(0,R_0+\ep)$.
In light of Lemma \ref{lem:jet-radial} in Appendix, for some $s\le0$, 
\begin{align*}
D\phi(x_0)=s\frac{x_0}{R_0+\ep}, \quad \text{and}\quad
\tr[\sig(D\phi(x_0))D^2\phi(x_0)]
\leq \frac{(n-1)s}{R_0+\ep}. 
\end{align*}
Thus, for $s>0$, 
\begin{align*}
&H(D\phi(x_0),D^2\phi(x_0)) -cI^\ep(R_0+\ep)
=-\tr[\sig(D\phi(x_0))D^2\phi(x_0)]-|D\phi(x_0)|\\
\geq&\, 
\frac{\big((n-1)-(R_0+\ep)\big)|s|}{R_0+\ep}\ge 0,
\end{align*}
which implies the conclusion. 
\end{proof}
We define
\[
 \psi(r):=\lim_{\ep \to 0} \psi^\ep(r)=\inf_{\ep \to 0} \psi^\ep(r) \quad
 \text{for} \ r\in\R. 
\]
 Actually, $\psi$ can be computed explicitly as following
\[
 \psi(r)=
\begin{cases}
 c\left( (r+(n-1)\log|r-n+1|) - (R_0+(n-1)\log|R_0-n+1|)\right) \quad &\text{for} \ r \in [0,R_0],\\
0 \quad &\text{for} \ r \in (R_0,\infty).
\end{cases}
\]
For $(x,t)\in \R^n \times [0,\infty)$, set
\begin{equation}\label{func:u-case1}
v(x,t):= \min\left\{\psi(|x|),ct\right\}=\lim_{\ep \to 0} u^\ep(x,t)=\inf_{\ep>0}u^{\ep}(x,t).   
\end{equation}
It is important noticing that $u^\ep$ converges to $v$ uniformly in $\R^n \times [0,\infty)$.

By the comparison principle and Lemma \ref{lem:psi-ep}, 
it is clear that $u^\ep \geq u$ and hence $v \geq u$.
Furthermore, $u^\ep$ is also a supersolution of (C) for all $\ep>0$, and so is $v$.
We now show that in fact $u=v$. In order to achieve this, we need the following result
\begin{lem}\label{lem:v-case1}
 The function $v$ is a subsolution of {\rm (C)}.
\end{lem}
\begin{proof}
Set $T_0:=(n-1) \log(n-1) - (R_0+(n-1)\log|R_0-n+1|)$.
For $t\geq T_0$, $v(x,t)=\psi(|x|)$ for all $x\in \R^n$ and there is nothing to check.

Let us now fix $(x_0,t_0)\in \R^n \times (0,T_0)$ such that $v(x_0,t_0)=\psi(x_0)=c t_0$.
Assume that $v-\phi$ has a strict maximum at $(x_0,t_0)$ 
for some test function $\phi \in C^2(\R^n \times (0,\infty))$ and that $\phi(x_0,t_0)=v(x_0,t_0)=c t_0$. 
Clearly, $0 \leq \phi_t(x_0,t_0) \leq c$.
Thanks to Lemma \ref{lem:jet-radial}, for some $s \in[\psi'(|x_0|),0]$, 
\begin{align*}
D\phi(x_0,t_0)=s\frac{x_0}{|x_0|}, \quad \text{and}\quad
\tr[\sig(D\phi(x_0))D^2\phi(x_0)] \geq 
\frac{(n-1)s}{|x_0|}.
\end{align*}
Thus, for $s>0$, 
\begin{equation}\label{subsln-1}
\phi_t(x_0,t_0)-H(D\phi(x_0),D^2\phi(x_0)) -c\mathbf{1}_E(x_0)\\
\leq  \phi_t(x_0,t_0) + \frac{s(|x_0|-n+1)}{|x_0|} -c.
\end{equation}

For $t<t_0$, let $r(t)$ be the function in $(|x_0|,R_0)$ which satisfies 
\[
 t=r(t)+(n-1)\log(n-1-r(t)) - (R_0 + (n-1)\log(n-1-R_0)), 
\]
and set $x(t):=r(t)x_0/|x_0|$. 
Then we have $\phi(x(t),t) \geq v(x(t),t)=ct$ for $t \leq t_0$ and $\phi(x(t_0),t_0)=\phi(x_0,t_0)=ct_0$.
Therefore,
\begin{align*}
 c &\geq \frac{d}{dt}(\phi(x(t),t))|_{t=t_0}=\phi_t(x_0,t_0) + D\phi(x_0,t_0)\cdot x'(t_0)\\
&=\phi_t(x_0,t_0) + r'(t_0) \left(D\phi(x_0,t_0)\cdot \frac{x_0}{|x_0|}\right)=\phi_t(x_0,t_0)+s\frac{|x_0|-n+1}{|x_0|}.
\end{align*}
We combine this and \eqref{subsln-1} to get the result.
\end{proof}

In conclusion, we obtain 
\begin{prop}\label{prop:case1}
Let $u$ be the maximal solution of {\rm(C)}. 
Then, we have the formula \eqref{func:u-case1}, and
thus $u$ is bounded on $\R^n\times[0,\infty)$. In particular, 
\[
 \lim_{t \to \infty} \frac{u(x,t)}{t}=0 \quad \text{in} \ C(\R^n). 
\]
\end{prop}
\begin{proof}
Since the maximal solution $u$ is obtained by \eqref{func:u-case1}, 
we have 
\[
u(x,t)=\min\{\psi(|x|), ct\} \leq \psi(|x|) \quad \text{for all} \ (x,t) \in \R^n \times [0,\infty).
\]
This immediately implies the conclusion. 
\end{proof}

\subsection{The case $R_0>n-1$}\label{subsec:supercritical}
Fix $\ep>0$.
We first look at an initial-boundary value problem of a linear partial differential equation 
in $1$-dimension: 
\[
{{\rm(L)}_\ep\quad}
\begin{cases}
\displaystyle
\varphi^\ep_t+\left(-\frac{(n-1)}{r}+1\right)\varphi^\ep_r=0 \quad &\text{in}\ (R_0+\ep,\infty)\times (0,\infty) \\
\varphi^\ep(R_0+\ep,t)=ct \quad &\text{on} \ [0,\infty).
\end{cases}
\]
By using the method of characteristics, we can find a solution to the above PDE 
\[
\varphi^\ep(r,t)=c\big(t-r-(n-1)\log(r-(n-1))+R_0+\ep+(n-1)\log(R_0+\ep-(n-1))\big).
\]
Define $u^\ep:\R^n \times[0,\infty) \to \R$ as 
\begin{equation*}
u^\ep(x,t):=
\left\{
\begin{array}{ll}
ct & \text{for all} \ (x,t)\in B(0,R_0+\ep) \times [0,\infty) \\
\left(\varphi^\ep(|x|,t)\right)_{+} & \text{for all} \ (x,t)\in (\R^n \setminus B(0,R_0+\ep)) \times [0,\infty).
\end{array}
\right. 
\end{equation*}
\begin{center}
\includegraphics[width=80mm]{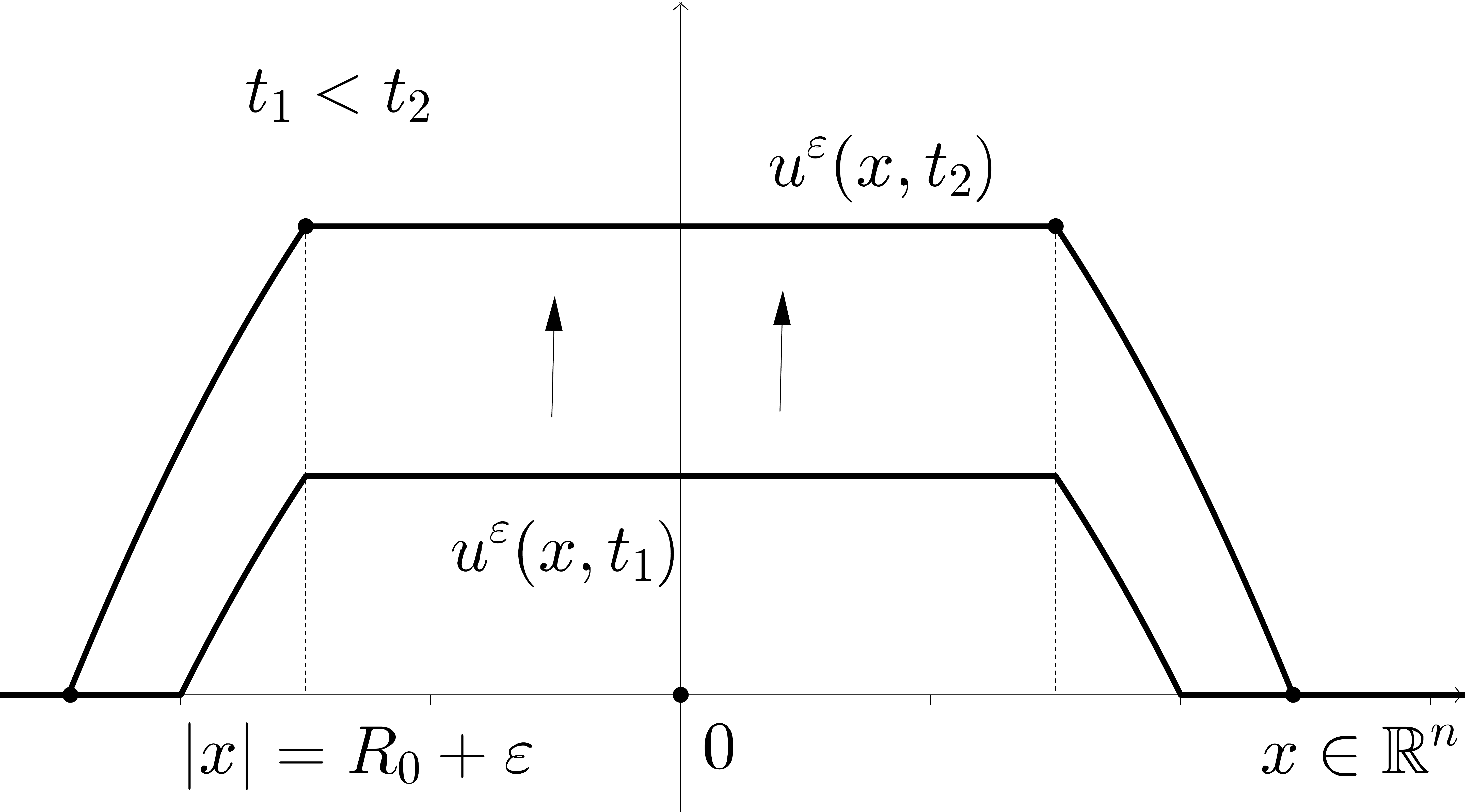}
\captionof{figure}{Picture of $u^\ep$ in Case 2}
\end{center}

\begin{lem}\label{lem:varphi-ep}
 The function $u^\ep$ is a viscosity supersolution of \eqref{eq:2} for $g(x)=c\mathbf{1}_{B(0,R_0+\ep)}$.
\end{lem}

\begin{proof}
 We only need to check at $(x_0,t_0) \in \R^n \times (0,\infty)$ where $u^\ep(x_0,t_0)=\varphi^\ep(|x_0|,t_0)=0$.
Assume that $u^\ep-\phi$ has a strict minimum at $(x_0,t_0)$ for some test function $\phi \in C^2(\R^n \times [0,\infty)$
and $u^\ep(x_0,t_0)=\phi(x_0,t_0)=0$.
In light of Lemma \ref{lem:jet-radial}, for some $s \in [\varphi_r(|x_0|,t_0),0]$, 
\begin{align*}
D\phi(x_0,t_0)=s\frac{x_0}{|x_0|}, \quad \text{and}\quad
\tr[\sig(D\phi(x_0,t_0))D^2\phi(x_0,t_0)|] \leq 
\frac{(n-1)s}{|x_0|}.
\end{align*}
Thus,
\begin{equation}\label{supersln-1}
\phi_t(x_0,t_0)+H(D\phi(x_0,t_0), D^2\phi(x_0,t_0))
-c\mathbf{1}_{B(0,R_0+\ep))}(x_0)
\geq \phi_t(x_0,t_0) + \frac{s(|x_0|-n+1)}{|x_0|}.
\end{equation}
For $t<t_0$, let $r(t)$ be the function in $(R_0+\ep,|x_0|)$ which 
satisfies
\[
 t=r(t)+(n-1)\log(r(t)-n+1) - (R_0+\ep + (n-1)\log(R_0+\ep-n+1)), 
\]
and set $x(t):=r(t)x_0/|x_0|$. 
Then we have $\phi(x(t),t) \leq u^\ep(x(t),t)=0$ for $t \leq t_0$ and $\phi(x(t_0),t_0)=\phi(x_0,t_0)=0$.
Therefore,
\begin{align*}
 0 &\leq \frac{d}{dt}(\phi(x(t),t))|_{t=t_0}=\phi_t(x_0,t_0) + D\phi(x_0,t_0)\cdot x'(t_0)\\
&=\phi_t(x_0,t_0) + r'(t_0) \left(D\phi(x_0,t_0)\cdot \frac{x_0}{|x_0|}\right)=\phi_t(x_0,t_0)+s\frac{|x_0|-n+1}{|x_0|}.
\end{align*}
We combine this and \eqref{supersln-1} to get the result.
\end{proof}
For $(x,t)\in \R^n \times [0,\infty)$, set
\begin{equation*}
v(x,t):= \lim_{\ep \to 0} u^\ep(x,t)=\inf_{\ep>0}u^{\ep}(x,t),  
\end{equation*}
It is important noticing that $u^\ep$ converges to $v$ locally uniformly in $\R^n \times [0,\infty)$.
We actually have 
\begin{equation}\label{func:u-case2}
v(x,t)=
\left\{
\begin{array}{ll}
ct & \text{for all} \ (x,t)\in B(0,R_0) \times [0,\infty) \\
(\varphi(|x|,t))_{+} & \text{for all} \ (x,t)\in (\R^n \setminus B(0,R_0)) \times [0,\infty),
\end{array}
\right. 
\end{equation}
where $\varphi$ is the solution to (L)$_0$, i.e., 
\[
\varphi(r,t)=c\big(t-r-(n-1)\log(r-(n-1))+R_0+(n-1)\log(R_0-(n-1))\big).
\]

By the comparison principle and Lemma \ref{lem:varphi-ep}, 
it is clear that $u^\ep \geq u$ and hence $v \geq u$.
Furthermore, $u^\ep$ is also a supersolution of (C) for all $\ep>0$, and so is $v$.
We now show that in fact $u=v$. In order to achieve this, we need the following result
\begin{lem}\label{lem:v-case2}
 The function $v$ is a subsolution of {\rm (C)}.
\end{lem}

\begin{proof}
It is enough to test at $(x_0,t_0) \in \R^n \times (0,\infty)$ in case $|x_0|=R_0$.
Assume that $v-\phi$ has a strict maximum at $(x_0,t_0)$ for some test function $\phi \in C^2(\R^n \times [0,\infty))$
and $v(x_0,t_0)=\phi(x_0,t_0)=ct_0$. We first note that $\phi_t(x_0,t_0)=c$.
We use Lemma \ref{lem:jet-radial} once more to deduce that 
for some $s \in [\varphi_r(R_0,t_0),0]$, 
\begin{align*}
D\phi(x_0,t_0)=s\frac{x_0}{R_0}, \quad \text{and} \ \quad
\tr[\sig(D\phi(x_0,t_0))D^2\phi(x_0,t_0)] \geq 
\frac{(n-1)s}{R_0}.
\end{align*}
Hence, for $s>0$, 
\[
\phi_t(x_0,t_0)+H(D\phi(x_0,t_0), D^2\phi(x_0,t_0))
-c\mathbf{1}_{E}(x_0) \leq   s\frac{R_0-n+1}{R_0} \leq 0.
\qedhere
\]
\end{proof}
In conclusion, we obtain 
\begin{prop}\label{prop:case2}
Let $u$ be the maximal solution of {\rm(C)}. 
Then, $u$ has the formula \eqref{func:u-case2}, and 
\[
\lim_{t \to \infty} \frac{u(x,t)}{t}=c \quad \text{locally uniformly for} \ x \in \R^n.
\]
\end{prop}

\subsection{The critical case $R_0=n-1$}
We denote by $u^r$ the maximal solution of (C) when $E=\ol{B}(0,r)$ when $r>n-1$.
By the comparison principle, we get that $u \leq u^r$ and hence
\[
u \leq \lim_{r \to 0} u^r = \inf_{r>0} u^r =ct \mathbf{1}_{\ol{B}(0,n-1)}=:v.
\]
It is clear that $v$ is a supersolution of (C). 
We now show that $v$ is in fact a subsolution of (C),
which yields again that $u=v=ct \mathbf{1}_{\ol{B}(0,n-1)}$.

\begin{lem}\label{lem:case3}
The function $v=ct \mathbf{1}_{\ol{B}(0,n-1)}$ is a subsolution of {\rm(C)}.
\end{lem}
\begin{proof}
As usual, it is enough to test at $(x_0,t_0) \in \R^n \times (0,\infty)$ in case $|x_0|=R_0=n-1$.
Assume that $v-\phi$ has a strict maximum at $(x_0,t_0)$ for some test function $\phi \in C^2(\R^n \times [0,\infty))$
and $v(x_0,t_0)=\phi(x_0,t_0)=ct_0$. We note first that $\phi_t(x_0,t_0)=c$.
Lemma \ref{lem:jet-radial} yields that for some $s\leq 0$, 
\begin{align*}
D\phi(x_0,t_0)=s\frac{x_0}{R_0}, \quad \text{and} \quad
\tr[\sig(D\phi(x_0,t_0))D^2\phi(x_0,t_0)] \geq 
\frac{(n-1)s}{R_0}=s. 
\end{align*}
Hence, for $s>0$, 
\[
\phi_t(x_0,t_0)+H(D\phi(x_0,t_0),D^2\phi(x_0,t_0))
-c\mathbf{1}_{E}(x_0) \leq   s-s = 0.
\qedhere
\]
\end{proof}
We conclude by the following result. 
\begin{prop}\label{prop:case3}
We have the following asymptotics
\[
\lim_{t \to \infty} \frac{u(x,t)}{t}=c \quad \text{uniformly for} \ x \in \ol{B}(0,n-1), 
\]
and
\[
\lim_{t \to \infty} \frac{u(x,t)}{t}=0 \quad \text{uniformly for} \ x \in \R^n \setminus \ol{B}(0,n-1).
\]
\end{prop}

\subsection{Some immediate consequences}
By using the above results, we get some results for general compact sets $E$ when it is either small enough or large enough.
\begin{cor}\label{prop:case4}
If $E \subset B(y,n-1)$ for some $y \in \R^n$, then
\[
\lim_{t \to \infty} \frac{u(x,t)}{t}=0 \quad \text{uniformly for} \ x \in \R^n.
\]
If $\ol{B}(y,n-1) \subset \inter E$ for some $y\in \R^n$, then 
\[
\lim_{t \to \infty} \frac{u(x,t)}{t}=c \quad \text{locally uniformly for} \ x \in \R^n.
\]
\end{cor}

\begin{cor}\label{thm:separate}
Assume that $E=\bigcup_{i=1}^k \ol{B}(y_i,r_i)$ for some given $k\in \N$, $y_i \in \R^n$, and $0<r_i \leq n-1$ for $1\leq i \leq k$.
Assume further that the closed balls $\ol{B}(y_i,r_i)$ for $1\leq i \leq k$ are disjoint. 
Denote by $K=\{i\,:\,1\leq i \leq k, \ r_i = n-1\}$. Then
\[
\lim_{t \to \infty} \frac{u(x,t)}{t}=c \quad \text{uniformly for} \ x \in \bigcup_{i \in K} \ol{B}(y_i,r_i)
\]
and
\[
\lim_{t \to \infty} \frac{u(x,t)}{t}=0 \quad \text{uniformly for} \ x \in \R^n \setminus \bigcup_{i \in K} \ol{B}(y_i,r_i).
\]
\end{cor}

\begin{proof}
For $1\leq i \leq k$, denote by $u^i$ the solution of (C) corresponding to $c\mathbf{1}_{\ol{B}(y_i,r_i)}$.
The important point is that $\{u^i >0\} \subset \ol{B}(y_i,r_i)$, which implies that 
\[
\{u_i>0\} \cap \{u_j>0\}=\emptyset \quad \text{for} \ i \neq j.
\]
Hence,  $u=\max_{1\leq i \leq k} u_i$. The results then follow from Propositions \ref{prop:case1} and \ref{prop:case3}. 
\end{proof}

\subsection{Optimal control interpretation and inhomogeneous $f$} \label{subsec:opt}
We now consider a slightly more general version of (C) in the spherical symmetric situation.
More precisely, we are concerned with \eqref{eq:2} in case $f(x)=h(|x|)$ where
$h:[0,\infty) \to [0,\infty)$ is upper semicontinuous and there exists $R>0$ such that
\[
{\rm supp}\ h \subset [0,R].
\]
It is again reasonable to look for radially symmetric solution $u(x,t)=\phi(|x|,t)=\phi(r,t)$. Then $\phi$ satisfies
\begin{equation}\label{eq:3}
\begin{cases}
\phi_t-\frac{n-1}{r}\phi_r -|\phi_r| = h(r) \quad &\text{in} \ (0,\infty) \times (0,\infty),\\
\phi(\cdot,0)=0 \quad &\text{on} \ [0,\infty).
\end{cases}
\end{equation}
Let $\psi=-\phi$, and then $\psi$ solves
\begin{equation}\label{eq:4}
\begin{cases}
\psi_t-\frac{n-1}{r}\psi_r +|\psi_r| +h(r)= 0 \quad &\text{in} \ (0,\infty) \times (0,\infty),\\
\psi(\cdot,0)=0 \quad &\text{on} \ [0,\infty).
\end{cases}
\end{equation}
The Hamiltonian of \eqref{eq:4} is $\tilde H(p,r)=|p|-\frac{n-1}{r} p + h(r)$, which is convex and singular at $r=0$.
We can easily compute the corresponding Lagrangian $\tilde L$ as
\[
\tilde L(q,r)=
\begin{cases}
-h(r) \quad &\text{if} \ \left|q+\frac{n-1}{r}\right| \leq 1\\
+\infty \quad &\text{otherwise}.
\end{cases}
\]
Therefore, the representation formula of $\psi$ in light of optimal control theory is
\[
\psi(r,t)=\inf \left\{ \int_0^t(-h(\gam(s)))\,ds\,:\, \gam([0,t]) \subset (0,\infty), \ \gam(t)=r,\ \left|\gam'(s)+\frac{n-1}{\gam(s)}\right| \leq 1 \ \text{a.e.} \right\}.
\]
This yields that, for $(r,t) \in (0,\infty) \times [0,\infty)$,
\begin{equation}\label{rep-phi}
\phi(r,t)=\sup \left\{ \int_0^t h(\gam(s))\,ds\,:\, \gam([0,t]) \subset (0,\infty), \ \gam(t)=r,\ \left|\gam'(s)+\frac{n-1}{\gam(s)}\right| \leq 1 \ \text{a.e.} \right\}.
\end{equation}
We  can then prove that $u(x,t):=\phi(|x|,t)$ is the maximal viscosity solution to (C) in a similar manner to 
that of Sections \ref{subsec:subcritical}--\ref{subsec:supercritical}. 
We now provide a general version of Proposition \ref{prop:case2}.

\begin{prop}\label{prop:h1}
Let $u$ be the maximal viscosity solution to {\rm(C)}. 
Assume that there exists $r_0>n-1$ such that
\[
c:=h(r_0)=\max_{r\in [0,\infty)} h(r).
\]
Then 
\[
\lim_{t\to \infty} \frac{u(x,t)}{t}=c \quad \text{locally uniformly for} \ x\in\R^n.
\]
\end{prop}

\begin{proof}
We only need to prove that $\phi(\cdot,t)/t\to c$ as $t\to\infty$ locally uniformly in $[0,\infty)$. 
It is clear from the representation formula that $\phi(r,t) \leq ct$ for $(r,t) \in (0,\infty) \times (0,\infty)$,
which yields that
\[
\limsup_{t\to \infty} \frac{\phi(r,t)}{t} \leq c \quad \text{uniformly for} \ r \in (0,\infty).
\]
We now need to obtain the lower bound. Fix $r \in (0,\infty)$ and consider two cases.

{\it Case 1: $r>r_0$.} Set $T_1:=\frac{r_0(r-r_0)}{r_0-(n-1)}$. For $t>T_1$, consider the curve $\gam:[0,t] \to (0,\infty)$ as 
\[
\gam(s):=
\begin{cases}
r_0 \quad &\text{for} \ 0<s<t-T_1,\\
r_0+(s-t+T_1)\frac{r_0-(n-1)}{r_0} \quad &\text{for} \ t-T_1<s<t.
\end{cases}
\]
One can check that $\gam$ is admissible in formula \eqref{rep-phi} and hence
\[
\phi(r,t) \geq \int_0^t h(\gam(s))\,ds \geq
\int_0^{t-T_1} h(\gam(s))\,ds \geq c(t-T_1),
\]
as $h$ is nonnegative, which is sufficient to get the conclusion.

{\it Case 2: $0 < r \leq r_0$.} 
We first consider the following ODE
\[
\begin{cases}
\xi'(s)=-1-\frac{n-1}{\xi(s)} \quad &\text{for}\ s>0,\\
\xi(0)=r_0.
\end{cases}
\]
Take $T_2>0$ to be the smallest value such that $\xi(T_2)=r$. It is immediate that $T_2 \leq r_0$.
For $t>T_2$, consider  $\gam:[0,t] \to (0,\infty)$ as 
\[
\gam(s):=
\begin{cases}
r_0 \quad &\text{for} \ 0\le s<t-T_2,\\
\xi(s-t+T_2) \quad &\text{for} \ t-T_2<s\le t.
\end{cases}
\]
Clearly, $\gam$ is admissible in formula \eqref{rep-phi} and 
\[
\phi(r,t) \geq \int_0^t h(\gam(s))\,ds \geq c(t-T).
\qedhere
\]
\end{proof}
Based on the above proposition and its proof, we have the following general result. 
\begin{thm}\label{thm:h1}
Set
\[
c_1:=\max_{r\in [n-1,\infty)} h(r) \quad \text{and} \quad c_2=\sup_{r \in (n-1,\infty)} h(r).
\]
Then,
\begin{align*}
&\lim_{t\to \infty} \frac{u(x,t)}{t}=c_1 \quad \text{uniformly for} \ x \in \ol{B}(0,n-1),\\
&\lim_{t\to \infty} \frac{u(x,t)}{t}=c_2 \quad \text{locally uniformly for} \ x \in \R^n\setminus\ol{B}(0,n-1).
\end{align*}
\end{thm}

The proof of this theorem is similar to that of Proposition \ref{prop:h1} hence omitted.
It is important noticing that the large time average result of this theorem 
covers that of all the cases in Propositions \ref{prop:case1}, \ref{prop:case2}, and \ref{prop:case3}.
It however does not give explicit/precise formulas of the maximal solution as in the mentioned propositions.

\begin{rem}\label{rem:general}
(i) One can easily generalize the surface evolution part to $V=v_\infty(\rho_c\kap+1)$, 
where $v_\infty, \rho_c>0$ are given constants. 
Indeed, after rescaling the values $x, t$, 
we can reduce the problem to the case where $V=\kap+1$. 
(ii) One can also deal with a general initial data which is bounded and uniformly 
continuous on $\R^n$ instead of the constant initial data. 
(iii) Theorem \ref{thm:h1} is surprising as even though the source term is very thin, 
it very essentially affects the growth of the crystal in a long time.  
More precisely, 
if we consider $f(x)=c\mathbf{1}_{\pl B(0,R_0)}$ for $R_0>n-1$ fixed, then we have 
\[
\lim_{t\to \infty} \frac{u(x,t)}{t}=c \quad \text{locally uniformly for} \ x \in \R^n.
\]
\end{rem}


\section{A framework to get estimates of growth rate}\label{sec:rough-est}
In a nonspherically symmetric case, it seems hard at this moment to obtain 
the precise large time average $\lim_{t\to \infty} u(\cdot,t)/t$, 
where $u$ is the maximal viscosity solution of (C). 
We will point out why so in Section \ref{sec:conclude}. 
Therefore, we here start to build up a framework to 
obtain rough estimates, namely 
the estimates on $\limsup_{t\to\infty}u(\cdot,t)/t$ and $\liminf_{t\to\infty}u(\cdot,t)/t$ first.

We first try to understand the behavior of the top and bottom of a solution $u$ to 
(C) as it should give an information of the behavior of the height of $u$.

\subsection{Motion of the top and the bottom of solutions} 
Let $v$ be a viscosity subsolution of (C). By the comparison principle, we have 
$v^\ast(x,t)\leq ct$ in $\mathbb{R}^n\times[0,\infty)$. For $t \geq 0$, set
\begin{equation}\label{set:Amax}
A_{\max}(t):=\left\{ x\in\mathbb{R}^n \,:\, v^\ast(x,t)=ct \right\}, 
\end{equation}
which is a compact set for $t>0$ since $v^\ast$ is upper semicontinuous and compactly supported.
\begin{lem}
Let $A_{\max}(t)$ be the set defined by \eqref{set:Amax}. 
Then $A_{\max}(t)$ is a set theoretic subsolution of $V=\kappa+1$, 
i.e., $h(x,t):=\mathbf{1}_{A_{\max}(t)}(x)$ is 
a viscosity subsolution of \eqref{eq:2} with $f=0$ 
{\rm (see \cite[Definition 5.1.1]{G} for details)}.
Moreover, $A_{\max}(t) \subset E$ for all $t\in(0,\infty)$. 
\end{lem}

\begin{proof}
We first notice that $v_c(x,t):=v(x,t)-ct$ is a viscosity 
subsolution of \eqref{eq:2} with 
the right hand side $f \equiv 0$,  
and $v_c\leq 0$ in $\mathbb{R}^n\times[0,\infty)$. Moreover, 
$A_{\max}(t)=\left\{ x\in\mathbb{R}^n\,:\, v_c^*(x,t)=0 \right\}$.
Thus, it is clear to see that $A_{\max}(t)$ is a set theoretic subsolution 
of $V=\kappa+1$ in view of \cite[Theorem 5.1.6]{G}.

We next prove that $A_{\max}(t) \subset E$ for all $t\in(0,\infty)$. 
Suppose otherwise that there would exist $x_0\in A_{\max}(t_0)\cap E^c$ for some $t_0>0$. 
Then $\varphi(x,t):=ct$ is a test function of $v^\ast$ from above. 
This is a contradiction as 
\[
c=\varphi_t(x_0,t_0)+H_{\ast}(D\varphi(x_0,t_0), D^2\varphi(x_0,t_0))\le c1_{E}(x_0)=0,
\] 
where $H$ is defined by \eqref{func:H}.
\end{proof}

Let $w$ be a viscosity supersolution of (C). 
By the comparison principle again, we have $w_*(x,t)\geq 0$ in $\mathbb{R}^n\times[0,\infty)$.
We set 
\begin{equation}\label{set:Amin}
A_{\min}(t):=\R^n \setminus \left\{x\in\mathbb{R}^n \,:\,w_*(x,t)=0\right\}=\{x\in \R^n\,:\,w_*(x,t)>0\}.
\end{equation}
\begin{lem}
Let $A_{\min}(t)$ be the set defined by \eqref{set:Amin}. 
Then $A_{\min}(t)$ is a set theoretic supersolution of $V=\kappa+1$, 
i.e., $h(x,t):=\mathbf{1}_{A_{\min}(t)}(x)$ is 
a viscosity supersolution of \eqref{eq:2} with $f=0$.
Moreover, ${\rm int} E \subset A_{\min}(t)$ for all $t\in(0,\infty)$. 
\end{lem}
\begin{proof}
We only prove ${\rm int} E \subset A_{\min}(t)$ for all $t\in(0,\infty)$. 
Suppose otherwise that there would exist $x_0\in E\cap A_{\min}^c(t_0)$ for some $t_0>0$. 
Then, $w(x_0,t_0)=0$ which implies $\varphi(x,t)\equiv 0$ is a test function of $w_\ast$ from below. 
This is a contradiction as 
\[
0=\varphi_t(x_0,t_0)+H^{\ast}(D\varphi(x_0,t_0), D^2\varphi(x_0,t_0))\ge c1_{E}(x_0)=c.
\qedhere 
\] \end{proof}

In this manner, $A_{\max}(t)$ and $A_{\min}(t)$ are a set theoretic subsolution and supersolution, 
respectively, of obstacle problems of $V=\kappa+1$ with 
\[
A_{\max}(t)\subset E, \ \text{and} \  
\inter E\subset A_{\min}(t)
\quad\text{for all} \ t\ge0. 
\]
We give here a level set formulation for later use. See \cite[Chapter 5]{G} for more details. 
The functions $\ol{h}(x,t):=\mathbf{1}_{A_{\max}(t)}(x), 
\ul{h}(x,t):=\mathbf{1}_{A_{\min}(t)}(x)$ are, respectively, a subsolution and a supersolution to 
\begin{align*}
&
\max\left\{
v_t-\left(\Div\left(\frac{Dv}{|Dv|}\right)+1\right)|Dv|, 
v-\mathbf{1}_{E}(x)\right\}=0
\quad \text{in} \ \R^n\times(0,\infty), \\ 
&
\min\left\{
v_t-\left(\Div\left(\frac{Dv}{|Dv|}\right)+1\right)|Dv|, 
v-\mathbf{1}_{\inter E}(x)\right\}=0
\quad \text{in} \ \R^n\times(0,\infty).
\end{align*}
We refer the readers to \cite{M, MT} for some of related works concerning asymptotic behavior of solutions of obstacle problems.  
Let us emphasize that even though \cite{MT} studies the large time behavior of obstacle problems for Hamilton-Jacobi equations with 
possibly degenerate diffusion $\tr(A(x)D^2u)$, our problem here is not included since the degeneracy of the diffusion depends on the gradient of the solution.

\subsection{Upper and lower estimates}
From the heuristic observation by the Trotter-Kato approximation, we realize that 
the motion of the top \eqref{set:Amax} and the bottom \eqref{set:Amin} can be described by the obstacle problem of the surface evolution equation.

For a closed set $A\subset\R^n$, we denote by $\cF^-[A](t)$ (resp., $\cF^+[A](t)$) 
the solution of the front propagation of the obstacle problem 
\[V=\kappa+1\quad \textrm{with obstacle} \ A, \textrm{i.e.,} \ 
\cF^-[A](t)\subset A \quad
{\rm(resp.,} \ \inter A  \subset \cF^+[A](t){\rm)}
\] 
for any $t\ge0$, and $\cF^\pm[A](0)=A$. 
We introduce two following geometric assumptions:
\begin{itemize}
\item[{\rm(G1)}]
there exist an open set $D \supset E$ and $t_0>0$ such that $\cF^-[\ol{D}](t_0)=\emptyset$,

\item[{\rm(G2)}]
$\cF^+[E](t) \to \R^n$ as $t \to \infty$,
\end{itemize}
where $E$ is the given set from the source.

It is worthwhile emphasizing here that 
for each $E\subset\R^n$ it is highly nontrivial to check whether 
(G1) and (G2) hold or not. 
This is a purely geometric problem which we have to investigate independently. 
In this subsection, we assume (G1) and (G2) first, and 
study how it gives an affect to the height of the solution to 
\eqref{eq:1}.
In the next section, we will discuss more on (G1) and (G2).


Recall that for any $T>0$, $u \in \USC(\R^n \times [0,T])$ and there exists $R_T>0$ such that 
\[
u(x,t)=0 \quad \text{for all}\ (x,t) \in (\R^n \setminus B(0,R_T)) \times [0,T].
\]
Thus, $x\mapsto u(x,T)$ attains its maximum at some point $x_0 \in B(0,R_T)$ and furthermore, in light of the subsolution built in \eqref{func:u-case1},
\[
u(x_0,T)=\max_{\R^n} u(\cdot,T)>0.
\]

\begin{lem}[Upper estimate]\label{lem:upper2} 
Assume {\rm(G1)} holds, and let $t_0$ be given by {\rm(G1)}. 
There exists $b\in(0,c)$ such that $\max_{x\in\R^n} u(x,t_0)\leq bt_0$. 
\end{lem}
\begin{proof}
Since $A_{\max}(t_0)=\emptyset$,
we have $\max_{x \in \R^n} u(x,t_0) < ct_0$. 
We set 
\[
b:=\max_{x\in \R^n} \frac{u(x,t_0)}{t_0}
\]
 to get the desired result.
\end{proof}


\begin{thm}[Global upper estimate]\label{thm:upper-esti}
Assume {\rm(G1)} holds, and let $t_0$ be given by {\rm(G1)}. 
There exists $b\in(0,c)$ such that 
\[
u(x,t) \leq bt + (c-b)t_0 \quad
\text{for all} \ (x,t) \in\mathbb{R}^n \times (0,\infty). 
\]
In particular, 
\[
\limsup_{t\to\infty}\left(\sup_{x\in \R^n} \frac{u(x,t)}{t}\right) \leq b.
\]
\end{thm}
\begin{proof}
Let $w$ be the maximal solution to \eqref{eq:2} with $g=\mathbf{1}_{\ol{D}}$. 
In light of Lemma \ref{lem:upper2}, there exists $b\in (0,c)$ such that $\max_{x\in \R^n} w(x,t_0) \leq b t_0$.

By the comparison principle, $u\leq w$ on $\R^n\times[0,\infty)$, which yields that 
\[
\max_{x\in \R^n} u(x,t_0) \leq \max_{x\in \R^n} w(x,t_0) \leq b t_0.
\]
Using again the comparison principle and induction, we deduce that $u(x,mt_0+t) \leq w(x,t)+mbt_0$ on $\R^n\times[0,\infty)$ for all $m\in \N$.
In particular, $u(x,mt_0) \leq mbt_0$ for any $x\in\R^n$ and $m\in\N$.

For $t\in\left(mt_0,(m+1)t_0\right)$, $m\in \N$, we observe that
\begin{align*}
u(x,t) & \leq u(x,mt_0)+c(t-mt_0)
\leq bmt_0 + c(t-mt_0)\\
& = bt + (c-b)(t-mt_0)
 \leq bt + (c-b)t_0,
\end{align*}
which gives us the estimate on $u(x,t)$ and also the estimate on $\limsup$.
\end{proof}

We get a lower bound in a similar manner. For $R \geq n$ and $t>0$, set
\[
 U_R(t):=\inf\{u_*(x,t)\,:\,x\in \ol{B}(0,R)\cup E\}.
\]
\begin{lem}[Lower estimate] 
Assume {\rm(G2)} holds. 
For $R \geq n$, there exist $a_R>0$ and $t_1>0$ such that 
$U_R(t_1) \geq a_R t_1$.
\end{lem}

\begin{thm}[Global lower estimate]\label{thm:lower-esti} 
Assume {\rm(G2)} holds. 
For $R \geq n$, there exist $a_R>0$ and $t_1>0$ such that 
\[
U_R(t) \geq a_R (t - t_1) \quad\text{for all}\quad t\ge0. 
\]
Thus, 
\[
\liminf_{t\to\infty}\left(\inf_{x\in B(0,R)} \frac{u_*(x,t)}{t}\right) \geq a_R.
\]
\end{thm}

\begin{rem}
By the propagation property it is not difficult to see that $a_R$ can be taken independent of $R$. 
Indeed, fix $R>n$ and $s>t_1$. For $t>0$, set 
\[
A_n(t):=\{x\in \R^n\,:\,u_*(x,t)>a_n (s-t_1)\}.
\]
Then $ B(0,n) \subset A_n(s)$.
Note furthermore that $A_n(t)$ is a set theoretic supersolution of $V=\kappa+1$.
Hence, there exists $t_2>0$ independent of $s$ such that
\[
 B(0,R) \subset A_n(s+t_2).
\]
In other words, $t_2$ is the time it takes to transform $B(0,n)$ into $B(0,R)$ under the 
forced mean curvature flow $V=\kappa+1$. We conclude that
\[
 u_*(x,t) \geq a_n(t-t_1-t_2) \quad \text{for all} \ (x,t) \in B(0,R) \times (0,\infty).
\]
Therefore, 
\[
\liminf_{t\to\infty}\left(\inf_{x\in B(0,R)} \frac{u_*(x,t)}{t}\right) \geq a_n.
\]
\end{rem}

\section{The behavior of $\cF^{\pm}[E](t)$} \label{sec:square}

In this section we investigate the precise behavior of the solution 
$\cF^\pm[E](\cdot)$ of the front propagation problems with obstacles.
We in particular consider a family of squares in $\R^2$, i.e.,   
\begin{equation}\label{E-sq}
E:=E(d)=\left\{ (x_1,x_2)\,:\, |x_i|\leq d,\ i=1,2 \right\}
 \end{equation}
for $d>0$ given as a nontrivial specific example. 
We first give a straightforward result of Proposition \ref{prop:case4}. 
\begin{prop}\label{prop:cor-square} 
The followings hold{\rm:}
\begin{itemize}
\item[(i)] If $d<1/\sqrt{2}$, then $u$ is bounded on $\R^2\times[0,\infty)$. 
In particular, $u(\cdot,t)/t\to 0$ uniformly in $\R^2$ as $t\to\infty$.  
\item[(ii)] If $d>1$, then $u(\cdot,t)/t\to c$ locally uniformly in $\R^2$ as $t\to\infty$. 
\item[(iii)] If $d=1$, then $u(\cdot,t)/t \to c$ uniformly on $\ol{B}(0,1)$ as $t\to\infty$.
\end{itemize}
\end{prop}
Note that in the case $d=1$, Proposition \ref{prop:case4} does not give the large time behavior
of $u(x,t)/t$ for $x\in \R^n \setminus \ol{B}(0,1)$.
As this is a critical case, we do not know yet how to handle this situation.

We now study the case where $1/\sqrt{2} < d < 1$, which is 
delicate. 
Our goal is to verify assumptions (G1) and (G2),
which in turn gives us some knowledge on the $\liminf$
and $\limsup$ behavior of $u(\cdot,t)/t$ as $t\to \infty$ 
by using Theorems \ref{thm:upper-esti}, \ref{thm:lower-esti}.


We first consider the behavior of $\cF^{-}[E](t)$. 
More precisely, we first study the behavior of the solution to the obstacle problem 
for the graph:
\begin{equation}\label{obstacle-1}
\displaystyle 
\max\left\{ y_t - \frac{y_{xx}}{1+(y_x)^2} - (1+(y_x)^2)^{1/2}, y-g(x) \right\}=0
\quad
\text{for} \ (x,t) \in (-D,D) \times (0,\infty), 
\end{equation}
where $g(x):=-|x|$, and $D:=\sqrt{2}d$. 
We construct a viscosity supersolution $w$ of \eqref{obstacle-1} in 
$[-D,D] \times [0,T]$
for $T>0$ to be chosen with 
$w(x,0) \geq g(x)-s$ in $[-D,D]$ for some $s>0$ satisfying $d+s<1$.
Define $W:[-D,D]\to\R$ by 
\[
W(x):=
\begin{cases}
-|x| \quad &\text{for} \ 1/\sqrt{2}\le |x|\le D\\
-\sqrt{2} + (1-x^2)^{1/2} \quad &\text{for} \ |x| \leq 1/\sqrt{2}, 
\end{cases}
\]
and set
\[
w(x,t)=\lam(t) W\left(\frac{x}{\lam(t)}\right), 
\]
where $\lam:[0,\infty)\to \R$ is the solution of the following ODE 
\begin{equation}\label{ODE-lam}
\begin{cases}
\displaystyle
\lam'(t)=\frac{1}{\lam(t)}-1 \quad \text{for} \ t>0,\\
\displaystyle
\lam(0)=\frac{s}{2}.
\end{cases}
\end{equation}
Pick $T>0$ to be the first time that $\lam(T)=d+s/\sqrt{2}<1$. Clearly for $t \in (0,T)$,
\[
\lam'(t) \geq \frac{1}{d+s/\sqrt{2}}-1 = \ep_0>0,
\]
Thus $T \leq \ep_0^{-1}<\infty$.
\begin{center}
\includegraphics[width=80mm]{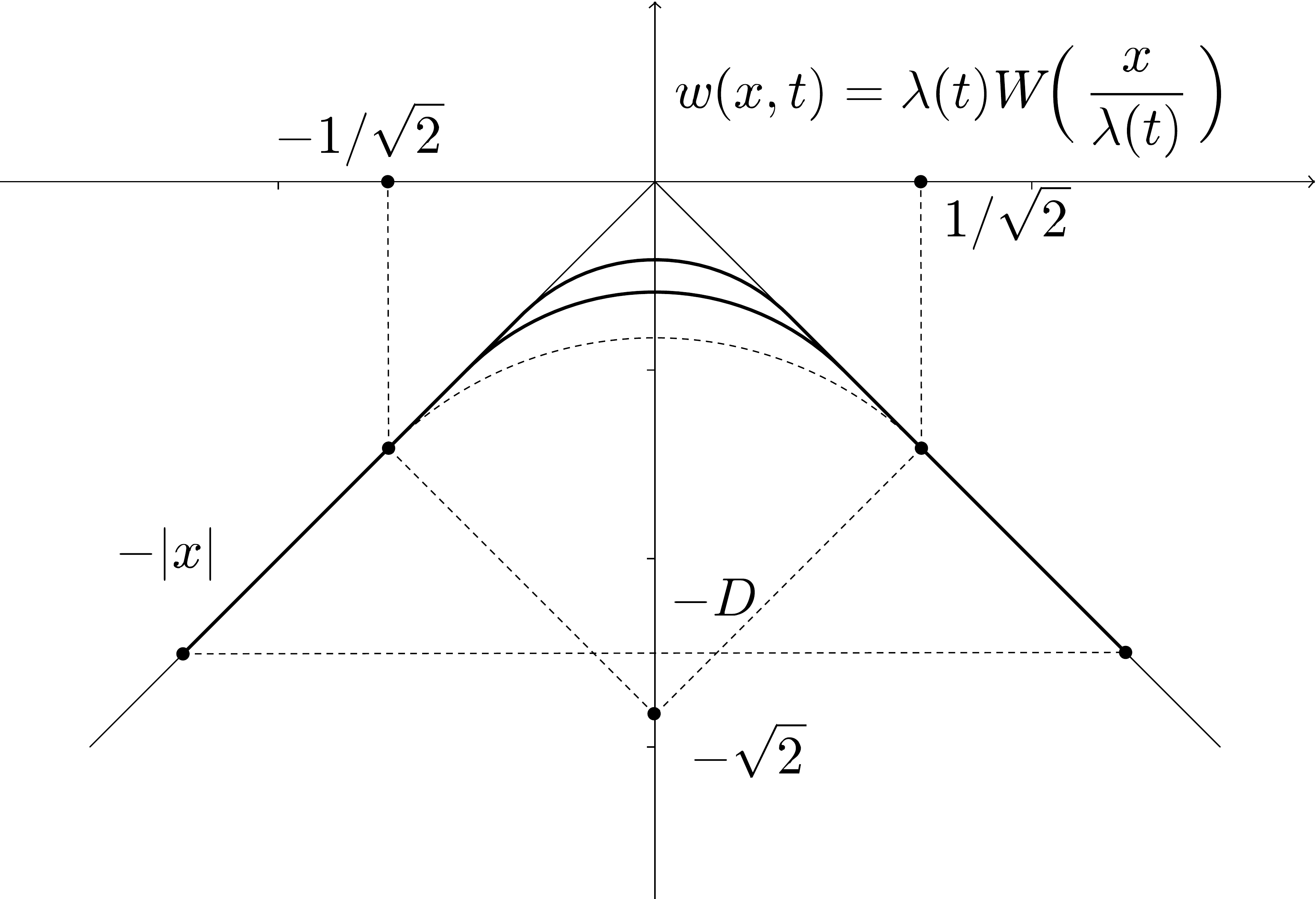}
\captionof{figure}{Graph of $w(x,t)$}
\end{center}

\begin{lem}\label{lem:G1}
The function $w$ defined above is a supersolution of \eqref{obstacle-1} in $[-D,D]\times [0,T]$ and $w(x,0) \geq g(x)-s$ in $[-D,D]$.
\end{lem}

\begin{proof}
It is straightforward from the definition that $w(x,0) \geq g(x)-s$.

Let $z=x/\lam(t)$. For $|z|<1/\sqrt{2}$, we compute that
\begin{align*}
w_t&=\lam'(t) (W- zW')=\lam'(t)\left(-\sqrt{2} + (1-z^2)^{1/2} - z \frac{-z}{(1-z^2)^{1/2}}\right)\\
&=\lam'(t)\left(-\sqrt{2}+  \frac{1}{(1-z^2)^{1/2}}\right),\\
w_x&=W'= \frac{-z}{(1-z^2)^{1/2}}, \qquad
w_{xx}= W'' \frac{1}{\lam(t)}=\frac{-1}{(1-z^2)^{3/2}} \frac{1}{\lam(t)}.
\end{align*}
By using the above, for $|z| < 1/\sqrt{2}$,
\begin{align*}
&w_t-\frac{w_{xx}}{1+(w_x)^2} - (1+(w_x)^2)^{1/2}\\
=\ & \lam'(t)\left(-\sqrt{2}+  \frac{1}{(1-z^2)^{1/2}}\right) + \frac{1}{\lam(t) (1-z^2)^{1/2}}-\frac{1}{(1-z^2)^{1/2}}\\
= \ &  \lam'(t)\left(-\sqrt{2}+  \frac{1}{(1-z^2)^{1/2}}\right) +\left(\frac{1}{\lam(t)}-1\right) \frac{1}{(1-z^2)^{1/2}}\\
= \ &\lam'(t)\left(-\sqrt{2}+  \frac{2}{(1-z^2)^{1/2}}\right) \geq 0,
\end{align*}
as $\lam'(t) \geq 0$ always for $t \in [0,T]$.

For $1/\sqrt{2}\le |z| \le D$, there is nothing to check as $w$ already touches the obstacle.
\end{proof}

\begin{thm}\label{thm:G1}
There exists $t_0>0$ such that $\cF^-[E](t)=\emptyset$ for all $t\ge t_0$.
\end{thm}
\begin{proof}
Recall that the original size of the obstacle (square) is $2d \in (\sqrt{2},2)$.
We consider the extended obstacle with size $2d+ \sqrt{2}s <2$. 
By Lemma \ref{lem:G1}, $\cF^-[E](T)$ is contained in a ball of radius $d+s/\sqrt{2}<1$.
We therefore deduce the existence of $t_0>0$ such that $\cF^-[E](t_0)=\emptyset$.
\end{proof}

\medskip
Next, we consider the following obstacle problem: 
\begin{equation}\label{obstacle-2}
\min\left\{ y_t - \frac{y_{xx}}{1+(y_x)^2} - (1+(y_x)^2)^{1/2}, y -g(x) \right\}=0 \quad \text{for} \ (x,t) \in (-D,D) \times (0,\infty).
\end{equation}
Pick $r \in (1,D)$.
We construct a subsolution $v$ of \eqref{obstacle-2} in $(-D,D) \times [0,T]$ for $T>0$
to be chosen such that 
\begin{align*}
&v(x,0)=g(x), \\
& 
(v(x,T)+D)^2+x^2 \geq r^2>1 \
&&\text{for all} \ x\in[-D,D].
\end{align*}

Let $\phi:(-r,r) \times [0,\infty) \to \R$ such that
\[
\phi(x,t) = - 2D + \sqrt{r^2-x^2} + \left(1-\frac{1}{r}\right) t.
\]
It is clear that $\phi$ is a separable subsolution of 
\[
\phi_t -  \frac{\phi_{xx}}{1+(\phi_x)^2} - (1+(\phi_x)^2)^{1/2} = 0 \quad \text{in} \ (-r,r) \times (0,\infty).
\]
Define the function $v:\R\times[0,\infty)\to\R$ by
\begin{equation*}
v(x,t):=
\left\{
\begin{array}{ll}
\max\{-|x|,\phi(x,t)\} &\text{for} \ |x| <r,\\
-|x| &\text{for} \ |x|\ge r.
\end{array}
\right.
\end{equation*}
\begin{center}
\includegraphics[width=80mm]{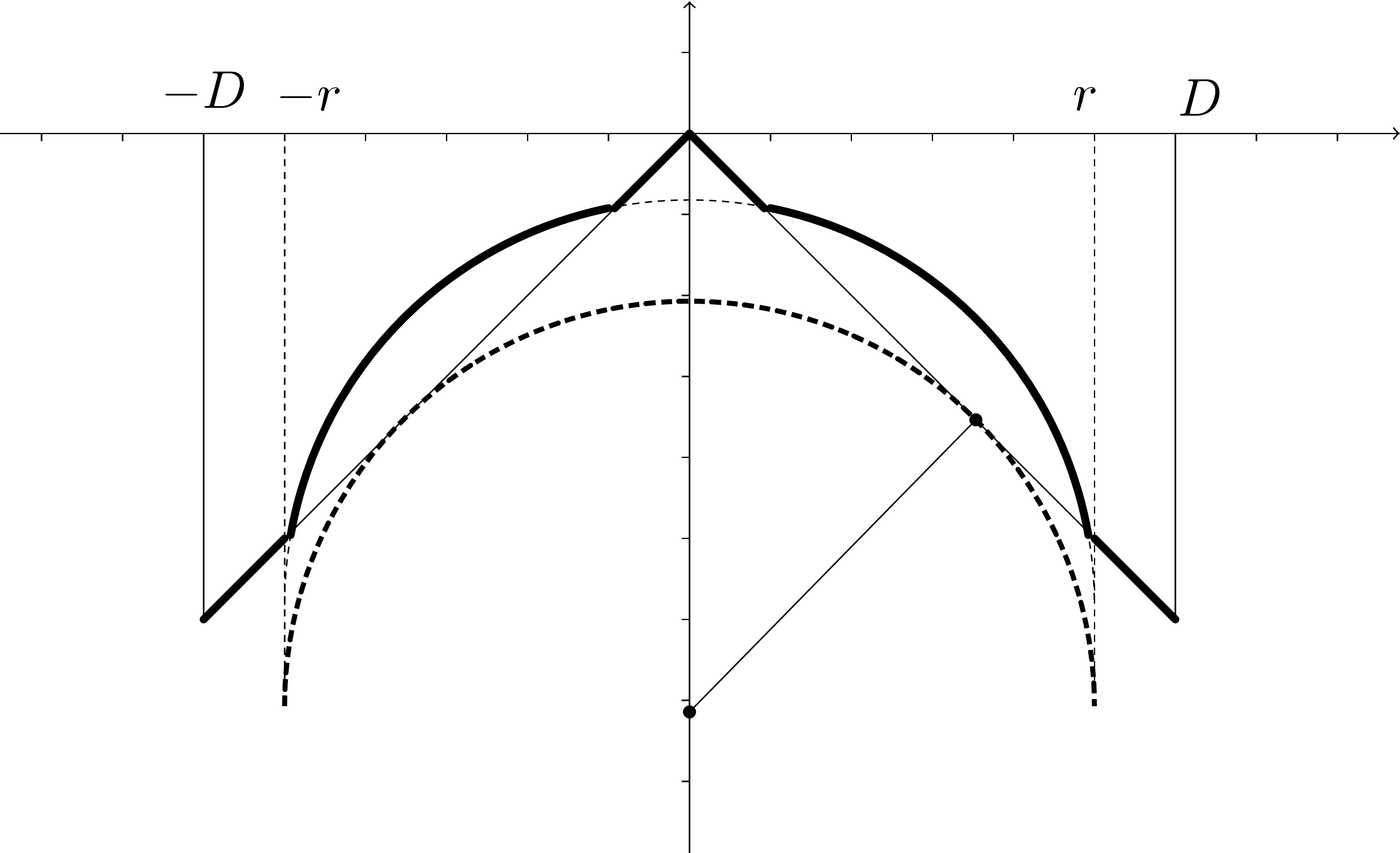}
\captionof{figure}{Picture of $v(x,t)$}\label{Fig4}
\end{center}
\begin{lem}\label{lem:G2}
 Let $T=\frac{r(2D-r)}{r-1}$.
Then the function $v$ defined as above is a subsolution of \eqref{obstacle-2} in $\R \times [0,T]$
with $v(x,0)=-|x|$ in $\R$ and
\begin{equation}\label{G2-grow}
(v(x,T)+D)^2+x^2 \geq r^2>1 \quad 
\text{for all} \ x\in[-D,D].
\end{equation}
\end{lem}

\begin{proof}
There is nothing to check in case $|x| \geq r$ as we always have $v(x,t)=-|x|$, which is the obstacle part.

We thus only need to check the case that $|x|<r$. In $(-r,r)\times (0,\infty)$, $v$ is defined as the maximum of
two subsolutions of \eqref{obstacle-2}, which implies that $v$ itself is also a subsolution.

We can easily check that at $T=\frac{r(2D-r)}{r-1}$,
\begin{equation*}
v(x,T):=
\left\{
\begin{array}{ll}
 -r +\sqrt{r^2-x^2} &\text{for} \ |x| <r,\\
-|x| &\text{for} \ |x|\ge r,
\end{array}
\right.
\end{equation*}
which yields \eqref{G2-grow} immediately.
\end{proof}

\begin{thm}\label{thm:G2}
 Assume that $1<D<\sqrt{2}$.
Then $\cF^+[E](t) \to \R^2$ as $t\to \infty$.
\end{thm}

\begin{proof}
 By Lemma \ref{lem:G2}, we have that $\cF^+[E](T)$ contains a ball of radius $r>1$. 
As this ball of radius $r$ expands to $\R^2$ under the forced mean curvature flow $V=\kappa+1$
as time goes to infinity, we get the conclusion.
\end{proof}

\begin{rem}
If we consider the case where $d<1/\sqrt{2}$, then 
we can easily check that 
the function $v$ defined by $v(x,t):=\sqrt{1-x^2}$ for all 
$(x,t)\in[-D,D]\times[0,\infty)$ is a supersolution of \eqref{obstacle-2} 
which is static. Therefore, the solution cannot grow up, which we have already 
got in Proposition \ref{prop:cor-square}. 
On the other hand, if we consider the case where $d>1/\sqrt{2}$, 
then we cannot avoid to have shocks as in Figure \ref{Fig4}, 
and therefore we cannot construct 
a static supersolution. 
\end{rem}

In light of Theorems \ref{thm:G1}, \ref{thm:G2}, we conclude that
\begin{prop} \label{prop:middle}
Assume that $1/\sqrt{2}<d<1$. Then there exists $\al, \beta$ such that $0<\al<\beta<c$ and
\[
\al \leq \liminf_{t \to \infty} \frac{u(x,t)}{t} \leq \limsup_{t \to \infty} \frac{u(x,t)}{t} \leq \beta \quad \text{locally uniformly for} \ x \in \R^2.
\]
\end{prop}

In a similiar manner to Theorem \ref{thm:G2}, we can get 
a slightly general result. 
\begin{prop}\label{prop:general}
Assume that, upon relabeling and reorienting the coordinates axes if necessary, there exist 
$l>1$ and functions $g_1, g_2 \in C([-l,l],\R)$ such that $g_1<g_2$ on $[-l,l]$ and
\[
\{(x,y)\,:\, x \in [-l,l] \quad \text{and} \quad g_1(x)\le y\le g_2(x) \} \subset E.
\] 
Then $\cF^{+}[E](t)\to\R^2$ as $t\to\infty$. 
\end{prop}

\begin{rem} \label{rem:two-cir}
One important problem in the crystal growth literature is to understand 
the large time average of the crystal growth with two sources as well as the interaction between the sources. 
In the case that the two sources are the same and of circular shape, the equation becomes
\begin{equation}\label{eq:two-source}
\begin{cases}
\displaystyle 
u_t-\left(\Div\left(\frac{Du}{|Du|}\right)+1\right)|Du|
=c \mathbf{1}_{B((a,0),R_0)}+c \mathbf{1}_{B((-a,0),R_0)}
 \quad &\text{in} \ \R^2\times(0,\infty), \\
u(\cdot,0)=0 \quad &\text{on} \ \R^2,  
\end{cases}
\end{equation}
where $a, R_0>0$ are given constants. 
We assume further that $R_0>a$, which means that the two sources overlap.
As a corollary to Proposition \ref{prop:general}, and Theorem \ref{thm:lower-esti}, 
we get 
\[
\liminf_{t\to\infty} \frac{u(x,t)}{t}\ge\al \quad \text{locally uniformly for} \ x \in \R^2,  
\]
for some $\al>0$
if and only if $a+R_0>1$. 
This is a tiny partial result toward this direction, which remains rather open so far.
At least, we are able to give a first condition to have a locally uniform growth as $t\to \infty$.
\end{rem}

\section{Conclusion}\label{sec:conclude}
We first established a well-posedness result 
for maximal viscosity solution of \eqref{eq:2} in Section \ref{sec:well}.
Note that we are not able to prove uniqueness of viscosity solutions because of the discontinuity of 
the right hand side of \eqref{eq:2} (see Remark 1).
We believe that the maximal solution is the correct physical solution.

In the spherically symmetric setting, we provide a complete analysis to understand the behavior
of the solution of (C) and its large time average in 
Propositions \ref{prop:case1}, \ref{prop:case2}, and \ref{prop:case3}.
We also study the case of inhomogeneous source $f$ of \eqref{eq:2} in Theorem \ref{thm:h1}.

In the nonspherically symmetric case, 
it seems extremely hard to obtain the precise large time average of the maximal viscosity solution of (C).  
For instance, if we consider the square case and use the Trotter-Kato product formula, 
then we could see clearly how complicated the behavior of the pyramid is as in Figure \ref{fig:square}.
This is of course  crucially different from that of the spherically symmetric case which we observed in Section \ref{sec:TK}.  
\begin{figure}[H]
  \begin{center}
    \begin{tabular}{c}
      \begin{minipage}{0.3\hsize}
        \begin{center}
          \includegraphics[clip, width=3cm]{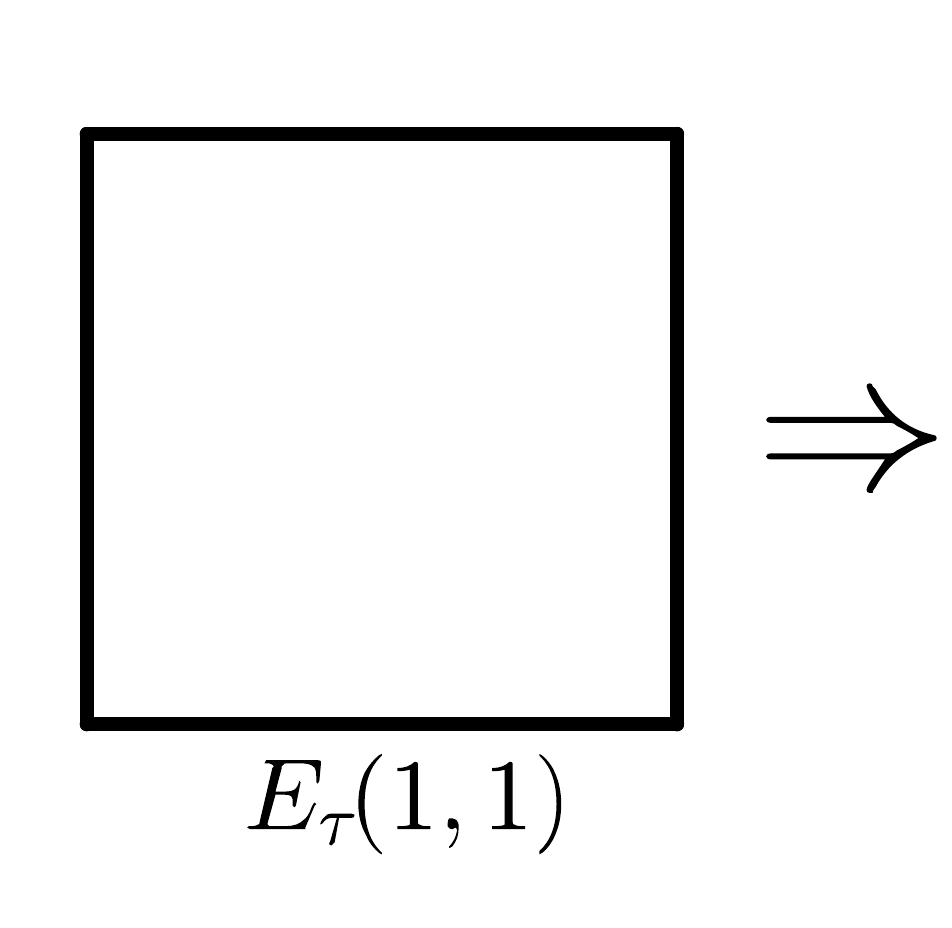}
        \end{center}
      \end{minipage}
      \begin{minipage}{0.3\hsize}
        \begin{center}
          \includegraphics[clip, width=3.5cm]{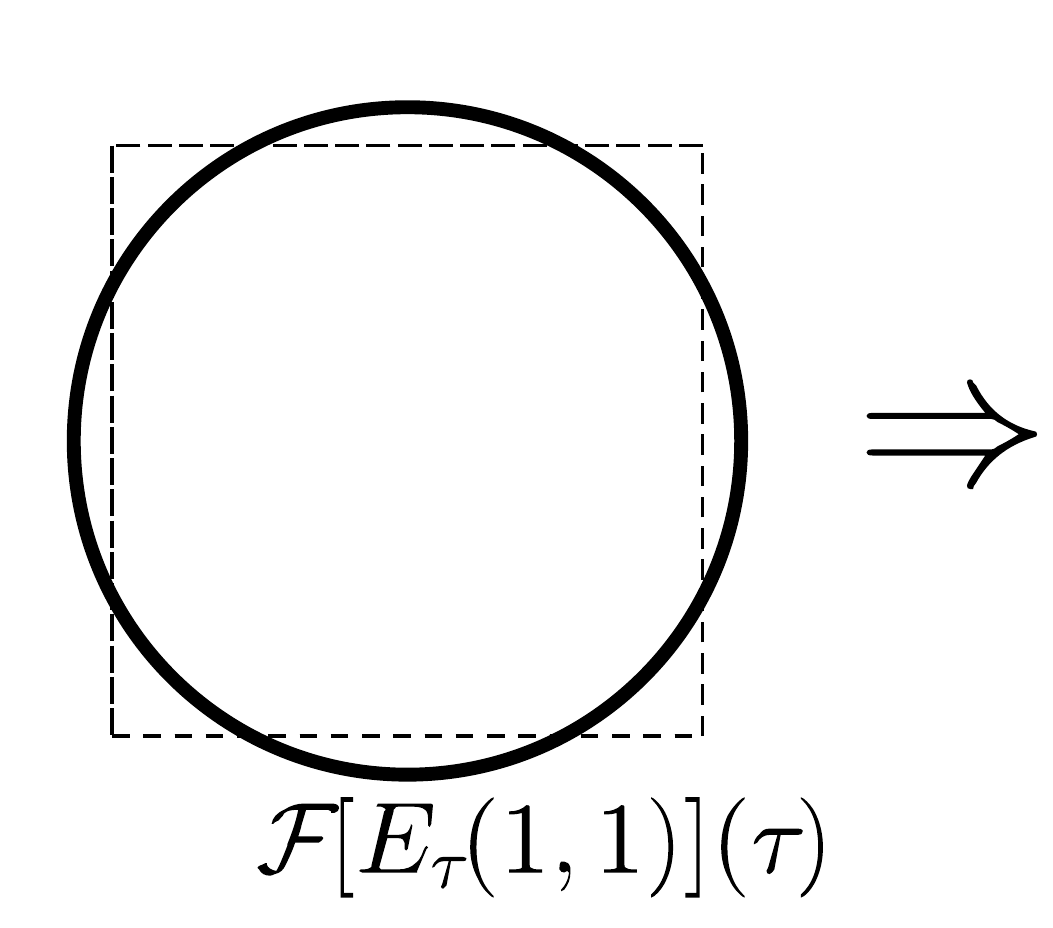}
          \end{center}
      \end{minipage}

      \begin{minipage}{0.3\hsize}
        \begin{center}
          \includegraphics[clip, width=3cm]{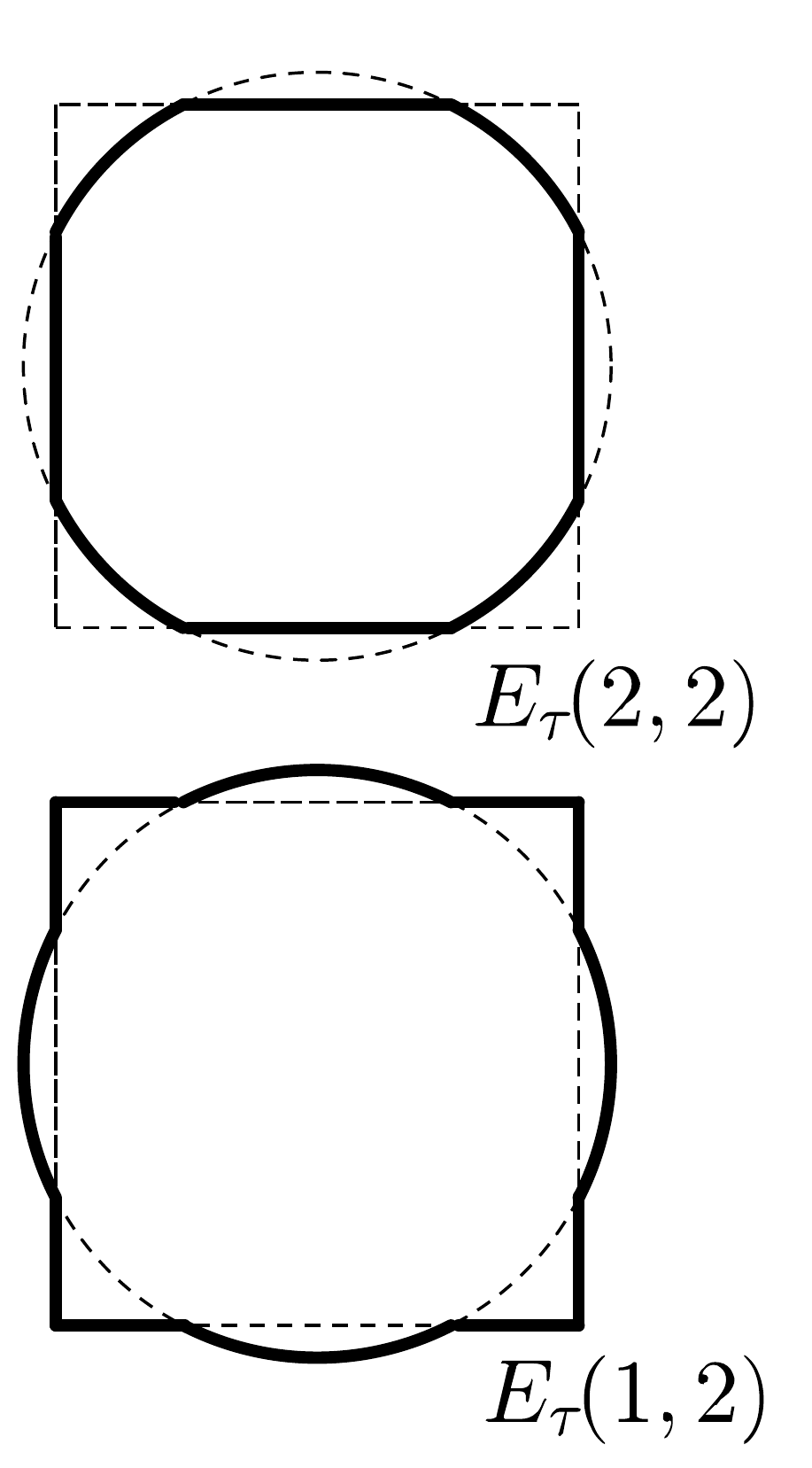}
        \end{center}
      \end{minipage}
    \end{tabular}
\caption{}
    \label{fig:square}
  \end{center}
\end{figure}
In the square case, we completely understood the case where $d<1/\sqrt{2}$ and $d>1$ 
in Proposition \ref{prop:cor-square} as a corollary of Proposition \ref{prop:case4}.
In the case where $1/\sqrt{2}<d<1$, we could only understand the behaviors of the top and the bottom of the pyramid and 
achieve $\liminf$ and $\limsup$ results in Proposition \ref{prop:middle}.
Note that we choose the square case here just to make it a clear representative.
Similar results hold for the rectangle and ellipse cases.

In order to understand the precise large time behavior of solutions, we 
somehow need to understand clearly the behavior of the middle layers of the pyramid.
So far we do not yet have tools to analyze these layers.

We are able to get a first nontrivial result for the case where $E$
is the union of the two balls of same size in Remark \ref{rem:two-cir}. 
The precise growth speed in this case however is still completely open.




\section{Appendix}
\begin{lem}\label{lem:jet-radial}
Let $\psi:[0,\infty) \to \R$ be a continuous function, which is $C^2$ in $(0,R) \cup (R,\infty)$ 
for some given $R>0$. Assume further that
\[
 \psi'(R-)=a \quad \text{and} \quad \psi'(R+)=b.
\]
The followings hold
\begin{itemize}
 \item[(i)] If $a<b$ then for any $\phi\in C^2(\R^n)$ such that $\psi(|x|)-\phi(x)$ has a strict minimum at 
$x_0 \in \partial B(0,R)$, then for some $s \in [a,b]$, 
\begin{align*}
D\phi(x_0)=s\frac{x_0}{R}, \quad \text{and} \quad
\tr[\sig(D\phi(x_0))D^2\phi(x_0)]\leq 
\frac{(n-1)s}{R}.
\end{align*}

\item[(ii)] If $a>b$ then for any $\phi\in C^2(\R^n)$ such that $\psi(|x|)-\phi(x)$ has a strict maximum at 
$x_0 \in \partial B(0,R)$, then for some $s \in [b,a]$, 
\begin{align*}
D\phi(x_0)=s\frac{x_0}{R}, \quad \text{and} \quad
\tr[\sig(D\phi(x_0))D^2\phi(x_0)]\geq 
\frac{(n-1)s}{R}. 
\end{align*}
\end{itemize}
\end{lem}

\begin{proof}
We only prove (i). Without loss of generality, assume $x_0=R e_1 = (R,0,\ldots,0)$ and $\psi(R)=0$.
We only need to consider $\phi$ of the quadratic form
\[
\phi(x)=p\cdot(x-x_0)+A(x-x_0)\cdot (x-x_0),
\]
where $A=(a^{ij})$ is a symmetric matrix. It is straightforward that
\[
D\phi(x_0)=p=s\frac{x_0}{R}=se_1 \quad \text{for some} \ s \in [a,b].
\]
We hence can rewrite $\phi(x) \leq 0$ for $|x|=R$ as
\begin{equation}\label{jet-1}
\phi(x)=
s(x_1-R) + a^{ij}(x_i-\del_{i1}R)(x_j-\del_{j1}R) \leq 0 \quad \text{for} \ x_1^2+\cdots x_n^2=R^2, 
\end{equation}
where $\del_{ij}=0$ if $i\not=j$, and $\del_{ii}=1$. 
Let $x_3=\cdots=x_n=0$ in the above to yield that, for $x_1^2+x_2^2=R^2$,
\[
s(x_1-R)+a^{11}(x_1-R)^2 + 2 a^{12}(x_1-R)x_2 + a^{22}(R^2-x_1^2) \leq 0,
\]
which can be simplified further as
\[
a^{11}(R-x_1) - 2 a^{12} x_2 + a^{22}(R+x_1) \leq s.
\]
Letting $x_1 \to R$ (which also means that $x_2 \to 0$) to deduce that
\begin{equation*}
2a^{22} \leq \frac{s}{R}.
\end{equation*}
By using \eqref{jet-1} in similar ways, we end up with
\begin{equation*}
2 \sum_{i=2}^n a^{ii} \leq \frac{(n-1)s}{R}.
\end{equation*}
Note finally that
\newcommand{\BigFig}[1]{\parbox{12pt}{\LARGE #1}}
\newcommand{\BigZero}{\BigFig{0}}
\[
\sig(p)=I-\frac{p\otimes p}{|p|^2}
={\scriptsize 
\left(
\begin{array}{ccccc}
0 & 0 &  \cdots & 0 \\
0 & 1& \cdots & 0 \\
\vdots & \vdots & \ddots &  \vdots\\
0 & 0 & \cdots & 1 
\end{array}
\right), 
}
\]
and thus
\[
\tr[\sig(D\phi(x_0))D^2\phi(x_0)]=\sum_{i=2}^n \phi_{x_i x_i}(x_0)=2 \sum_{i=2}^n a^{ii} \leq \frac{(n-1)s}{R}.
\qedhere
\]
\end{proof}

\bibliographystyle{amsplain}
\providecommand{\bysame}{\leavevmode\hbox to3em{\hrulefill}\thinspace}
\providecommand{\MR}{\relax\ifhmode\unskip\space\fi MR }
\providecommand{\MRhref}[2]{%
  \href{http://www.ams.org/mathscinet-getitem?mr=#1}{#2}
}
\providecommand{\href}[2]{#2}

\end{document}